\documentclass[11pt,a4paper]{amsart}
\usepackage[T1]{fontenc}
\usepackage[utf8]{inputenc}
\usepackage{lmodern}
\usepackage{amsmath,amssymb,amsthm}
\usepackage[margin=28mm]{geometry}
\usepackage{microtype}
\usepackage{tikz-cd}
\usepackage[normalem]{ulem}
\usepackage{xcolor}
\usepackage[colorlinks=true, pagebackref=true, linkcolor=red, citecolor=blue, urlcolor=blue]{hyperref}
\usepackage[alphabetic,backrefs]{amsrefs} 
\renewcommand*{\backref}[1]{}
\renewcommand*{\backrefalt}[4]{%
    \ifcase #1 %
    \or
        $\uparrow$ \textcolor{red}{#2}%
    \else
        $\uparrow$ \textcolor{red}{#2}%
    \fi
}
\allowdisplaybreaks
\newtheorem{theorem}{Theorem}[section]
\newtheorem{proposition}[theorem]{Proposition}
\newtheorem{lemma}[theorem]{Lemma}
\newtheorem{corollary}[theorem]{Corollary}
\theoremstyle{definition}
\newtheorem{definition}[theorem]{Definition}
\newtheorem{example}[theorem]{Example}

\newcommand{\C}{\mathbb C}
\newcommand{\R}{\mathbb R}
\newcommand{\Q}{\mathbb Q}
\newcommand{\Z}{\mathbb Z}
\newcommand{\ii}{\mathbf i}
\renewcommand{\d}{\mathrm d}
\newcommand{\db}{\overline\partial}
\newcommand{\g}{\mathfrak g}
\newcommand{\h}{\mathfrak h}
\newcommand{\LL}{\mathfrak L}
\newcommand{\UU}{\mathfrak U}
\newcommand{\br}[2]{\llbracket #1,#2\rrbracket}
\newcommand{\llbracket}{[\![}
\newcommand{\rrbracket}{]\!]}
\DeclareMathOperator{\Ann}{Ann}
\DeclareMathOperator{\ad}{ad}
\DeclareMathOperator{\tr}{tr}
\DeclareMathOperator{\spec}{spec}
\DeclareMathOperator{\im}{im}
\DeclareMathOperator{\coker}{coker}
\DeclareMathOperator{\id}{id}

\title{On the generalized cohomologies of almost abelian solvmanifolds}
\author[Andrade]{Mauro de Andrade Pinto}
\address{Instituto de Matem\'atica, Estat\'istica e Computa\c c\~ao Cient\'ifica,
Universidade Estadual de Campinas, Campinas, Brazil}
\email{mauroap1415@gmail.com}
\author[Campon\^es do Brasil]{Let\'icia Campon\^es do Brasil Maia}
\address{Instituto de Matem\'atica, Estat\'istica e Computa\c c\~ao Cient\'ifica,
Universidade Estadual de Campinas, Campinas, Brazil}
\email{leticiacampones@gmail.com}
\author[Cavenaghi]{Leonardo F. Cavenaghi}
\address{Institute of Mathematics and Informatics, Bulgarian Academy of Sciences,
Sofia, Bulgaria}
\email{leonardofcavenaghi@gmail.com}
\author[Martins]{Pedro Antonio Muniz Martins}
\address{Instituto de Matem\'atica, Estat\'istica e Computa\c c\~ao Cient\'ifica,
Universidade Estadual de Campinas, Campinas, Brazil}
\email{pedroa.muniz9@gmail.com}
\date{}

\begin{document}
\begin{abstract}
Generalized complex structures give rise to generalized Dolbeault,
Bott--Chern and Aeppli cohomologies. We study these cohomologies for
invariant generalized complex structures on almost abelian Lie algebras
$\mathfrak g_A=\mathbb R e_0\ltimes_A\mathfrak h$, where $\mathfrak h$ is
an abelian ideal of codimension one, and on almost abelian solvmanifolds
$\Gamma\backslash G_A$. We first show that the invariant generalized Dolbeault
cohomology can be identified with the Chevalley--Eilenberg cohomology of the complex Lie
algebra $\mathcal L$, the $+{\mathbf i}$-eigenspace of the structure, with
coefficients in a character determined by the canonical line. We prove that $\mathcal L$ is itself an almost abelian Lie algebra; thus the cohomology reduces to the kernel and cokernel of an explicit derivation. Assuming that $A$ is diagonalizable over $\mathbb R$, we obtain a combinatorial formula for the cohomology dimensions and compare with compact quotients. Other comparison results are discussed, noting that the cohomology of invariant forms always injects into that of $\Gamma\backslash G_A$.
\end{abstract}
\maketitle

\section{Introduction}

Generalized complex geometry, introduced by Hitchin and developed by
Gualtieri and Cavalcanti, provides a unified framework for complex and
symplectic geometry \cites{Hitchin2003,Gualtieri2011,Cavalcanti2004,Cavalcanti2005}. Let $X$ be a
smooth manifold of real dimension $2n$, and endow $T_X\oplus T^*_X$ with
the split-signature pairing
\[
\langle Y+\xi,W+\eta\rangle=\tfrac12\bigl(\xi(W)+\eta(Y)\bigr).
\]
A \emph{generalized almost complex structure} on $X$ is an endomorphism
$\mathcal J$ of $T_X\oplus T^*_X$ that is orthogonal for this pairing and
satisfies $\mathcal J^2=-\id$. Its $+\ii$-eigenbundle
\[
L_X=\ker(\mathcal J-\ii\id)\subset(T_X\oplus T^*_X)\otimes_\R\C
\]
satisfies $(T_X\oplus T^*_X)\otimes_\R\C=L_X\oplus\overline{L_X}$.
Orthogonality forces $L_X$ to be isotropic. Since its complex rank is $2n$, it
is maximal isotropic. The structure is \emph{integrable} if the space of
sections $C^\infty(L_X)$ is closed under the complexified Courant
bracket. It is then called a \emph{generalized complex structure}, and
$X$ a \emph{generalized complex manifold}.

The complexified bundle $(T_X\oplus T^*_X)\otimes_\R\C$ acts on complex
differential forms by Clifford multiplication,
$c(Y+\xi)\alpha=\iota_Y\alpha+\xi\wedge\alpha$. Maximal isotropy implies
that $L_X$ is the Clifford annihilator
\[
L_X=\{v\in(T_X\oplus T^*_X)\otimes_\R\C:\ c(v)K_X=0\}
\]
of a unique line subbundle
$K_X\subset\bigwedge^\bullet T^*_X\otimes_\R\C$. A line whose
annihilator is maximal isotropic is a \emph{pure spinor line}, and we
call $K_X$ the \emph{canonical line bundle} of $\mathcal J$. Around each
point it has a smooth local generator $\rho$, and at each point $x$
\[
\rho_x=c_x\,e^{\beta+\ii\omega}\wedge\Omega,\qquad
\Omega=\theta_1\wedge\cdots\wedge\theta_k,\qquad
\omega^{n-k}\wedge\Omega\wedge\overline\Omega\ne0,
\]
where $c_x\in\C^*$, $\beta$ and $\omega$ are real two-forms, and
$\theta_1,\ldots,\theta_k$ are complex covectors, with $\Omega=1$ when
$k=0$. The integer $k$ is the \emph{type} of $\mathcal J$ at $x$.
Complex structures have type $n$, and symplectic structures have type
$0$.

Following \cites{Gualtieri2011, Cavalcanti2006}, a generalized almost complex structure
decomposes complex forms as
\[
\bigwedge^\bullet T^*_X\otimes_\R\C=\bigoplus_{j=-n}^nU^j,\qquad
U^j=\bigwedge^{n-j}\overline{L_X}\cdot K_X,
\]
where the dot denotes Clifford multiplication. Let $U^j(X)$ be the space
of smooth sections of $U^j$. Integrability is equivalent to
$\d U^j(X)\subseteq U^{j-1}(X)\oplus U^{j+1}(X)$, so that
$\d=\partial+\db$ with
$\partial^2=\db^2=\partial\db+\db\partial=0$. The bi-differential
$\Z$-graded complex $(U^\bullet(X),\partial,\db)$ defines the generalized
Dolbeault cohomologies $GH^j_{\db}(X)$ and $GH^j_\partial(X)$. Furthermore, extending the framework advanced in \cite{Angella2012} through the approach of \cite{ACK} yields the generalized Bott-Chern and Aeppli cohomologies
\[
GH^j_{BC}(X)=\frac{\ker\partial\cap\ker\db}{\im\partial\db},\qquad
GH^j_A(X)=\frac{\ker\partial\db}{\im\partial+\im\db},
\]
all computed on $U^j(X)$.

In this paper we study these cohomologies for invariant generalized
complex structures on almost abelian Lie algebras and almost abelian
solvmanifolds. A \emph{solvmanifold} is a compact quotient
$\Gamma\backslash G$, where $G$ is a simply connected solvable Lie group
and $\Gamma<G$ is a lattice, that is, a discrete cocompact subgroup. An
\emph{almost abelian} Lie algebra has the form
\[
\g_A=\R e_0\ltimes_A\h,
\]
where $\h$ is an abelian ideal of codimension one and the entire Lie bracket structure is
encoded in the matrix $A$ of $\ad_{e_0}|_\h$. Let $G_A$ be the simply
connected Lie group with Lie algebra $\g_A$. An \emph{almost abelian
solvmanifold} is a quotient $\Gamma\backslash G_A$. Left-invariant
generalized complex structures on $G_A$ descend to $\Gamma\backslash G_A$,
and we call the resulting structures \emph{invariant}. They are
determined by linear algebra on $\g_A\oplus\g_A^*$, which we now
describe.

Let $\g$ be a real Lie algebra of dimension $2n$. A generalized complex
structure on $\g$ is an endomorphism $\mathcal J$ of $\g\oplus\g^*$ with $\mathcal J^2=-\id$,
orthogonal for the natural pairing, whose $+\ii$-eigenspace
\[
\LL=\ker(\mathcal J-\ii\id)\subset(\g\oplus\g^*)_\C
\]
is closed under the invariant Courant bracket. Then $\LL$ is a complex
Lie algebra of dimension $2n$. Clifford multiplication acts on
$\bigwedge^\bullet\g^*_\C$, and the canonical line
\[
K=\{\alpha\in\bigwedge^\bullet\g^*_\C:\ c(v)\alpha=0
\text{ for every }v\in\LL\}
\]
is one-dimensional. A nonzero generator $\rho\in K$ is a pure spinor
defining $\mathcal J$. The invariant forms
$\UU^j=\bigwedge^{n-j}\overline\LL\cdot K$ form a bi-differential
complex, whose cohomologies we denote by $GH^j_{\db}(\g)$,
$GH^j_\partial(\g)$, $GH^j_{BC}(\g)$ and $GH^j_A(\g)$. This paper has
three goals, which we describe next.

\

Our first goal is to identify the invariant generalized Dolbeault complex
with a Chevalley--Eilenberg complex of the Lie algebra $\LL$
(Section~\ref{sec:spinors}). Clifford multiplication by $\LL$ annihilates
$K$, so it cannot provide the coefficient action. Instead, let $\LL$ act on
$K$ by $v\cdot s=c(v)\d s$. Since $K$ is a line, this action is given
by a character $\chi\in\LL^*$:
\[
\chi(v)\rho:=c(v)\d\rho\qquad(v\in\LL).
\]
The character does not depend on the choice of $\rho$, and it vanishes
exactly when $\rho$ is closed. Denote by $\C_\chi $ the one-dimensional
module on which $v$ acts by multiplication by $\chi(v)$. We prove
(Proposition~\ref{prop:canonical-complex}) that
\[
GH_{\db}^{n-p}(\g)\cong H^p(\LL,\C_\chi),\qquad
GH_{\partial}^{-n+p}(\g)\cong H^p(\overline\LL,\C_{\bar\chi}),
\]
where the right-hand sides are Chevalley--Eilenberg cohomologies with the
indicated coefficients. This is the invariant counterpart of Gualtieri's
description of the generalized Dolbeault complex as the Lie algebroid
complex of $L_X$ with coefficients in $K_X$ \cite{Gualtieri2011}; see
\eqref{eq:global-ACK}. Its advantage is that the twisting by the
canonical line becomes an explicit character.

For an almost abelian Lie algebra $\g_A$, the Lie algebra $\LL$ is again
almost abelian (Proposition~\ref{thm:L-almost-abelian}). The map
$\phi:\LL\to\C$, $\phi(X+\xi)=e^0(X)$, is a surjective homomorphism with
abelian kernel $\mathfrak n$. Hence, for any $Z\in\LL$ with $\phi(Z)=1$,
\[
\LL=\C Z\ltimes_C\mathfrak n,\qquad
C=\ad_Z|_{\mathfrak n}:\mathfrak n\to\mathfrak n,\qquad
\dim_\C\mathfrak n=2n-1,
\]
and the canonical character is $\chi=\tau\phi$ with $\tau=\chi(Z)$.
Consequently, the complex $C^\bullet(\LL,\C_\chi)$ splits into two-term
complexes, and (Lemma~\ref{lem:wang})
\[
GH_{\db}^{n-p}(\g_A)\cong\ker\Psi_p\oplus\coker\Psi_{p-1},\qquad
\Psi_p=\Psi_C|_{\bigwedge^p\mathfrak n^*}-\tau\id
\]
where $\Psi_C$ is the derivation of $\bigwedge^\bullet\mathfrak n^*$
extending $\eta\mapsto\eta\circ C$. If $c_1,\ldots,c_{2n-1}$ are the
eigenvalues of $C$, the eigenvalues of $\Psi_p$ are the sums
$\sum_{i\in I}c_i-\tau$ over subsets $I$ of size $p$. These statements
hold for every invariant generalized complex structure on every almost
abelian Lie algebra.

\

Our second goal is to make these formulas explicit when $A$ is
diagonalizable over $\R$ (Section~\ref{sec:diagonal}). In this case,
\cite{PaperI} characterizes the existence of invariant generalized
complex structures in terms of the spectrum of $A$
(Theorem~\ref{thm:diagonal-existence}): $\g_A$ admits one of type $k$ if
and only if the eigenvalues of $A$, with multiplicities, can be arranged
as
\[
\spec(A)=\{\lambda_1,\lambda_1,\ldots,\lambda_h,\lambda_h,
\mu,\zeta_1,-\zeta_1,\ldots,\zeta_s,-\zeta_s\},\qquad h+s=n-1,
\]
with $k=h$ or $k=h+1$. The identical pairs supply complex directions, and
the opposite pairs supply symplectic planes. The two remaining directions,
spanned by $e_0$ and an eigenvector for $\mu$, form a symplectic plane in
Model~I ($k=h$) and a complex line in Model~II ($k=h+1$);
see~\eqref{eq:models}. For these structures, $C$ is diagonalizable with
\[
\spec(C)=\{\pm\lambda_1,\ldots,\pm\lambda_h,
\pm\zeta_1,\ldots,\pm\zeta_s,\nu\},
\]
where $\nu=\mu$ in Model~I and $\nu=-\mu$ in Model~II. Moreover,
$\d\rho=-\Lambda e^0\wedge\rho$ and $\chi=-\Lambda\phi$, where
$\Lambda=\sum_i\lambda_i$ in Model~I and $\Lambda=\mu+\sum_i\lambda_i$ in
Model~II (Lemma~\ref{lem:model-integrability} and
Proposition~\ref{prop:specC}).

Let $c_1,\ldots,c_{2n-1}$ be the eigenvalues of $C$, listed with
multiplicity, so that repeated eigenvalues are distinguished by
their indices. We prove (Theorem~\ref{thm:master}) that
\[
\dim GH_{\db}^{n-p}(\g_A)=N_p+N_{p-1},\qquad
N_p=\#\Bigl\{I\subseteq\{1,\ldots,2n-1\}:
|I|=p,\ \sum_{i\in I}c_i=-\Lambda\Bigr\},
\]
with $N_p=0$ outside $0\le p\le2n-1$. Equivalently,
\[
\sum_{p=0}^{2n}\dim GH_{\db}^{n-p}(\g_A)\,t^p
=(1+t)\,[q^{-\Lambda}]\prod_{i=1}^{2n-1}\bigl(1+tq^{c_i}\bigr),
\]
where $[q^{-\Lambda}]$ extracts the coefficient of $q^{-\Lambda}$. If
$\tr A=0$, then $\dim GH_{\db}^j(\g_A)=\dim GH_{\db}^{-j}(\g_A)$.
Suppose, in addition, that the indexed family
$\lambda_1,\ldots,\lambda_h,\zeta_1,\ldots,\zeta_s$ is linearly
independent over $\Q$. Then, for either model
(Corollary~\ref{cor:generic}),
\[
N_{h+2r}=N_{h+2r+1}=\binom{s}{r}\quad(0\le r\le s),
\qquad N_p=0\text{ otherwise},
\]
so that
\[
\sum_{p=0}^{2n}\dim GH_{\db}^{n-p}(\g_A)\,t^p=t^h(1+t)^2(1+t^2)^s .
\]
In particular, the total dimension is $2^{s+2}$, and
$GH_{\db}^j(\g_A)=0$ for $|j|>n-h$. 

\

Our third goal is to decide when these invariant computations give the
cohomology of a compact quotient (Section~\ref{sec:quotients}). Let
$\Gamma<G_A$ be a lattice -- whose existence is discussed in  \cite{Andrada2023}*{Proposition 2.5} -- and $X=\Gamma\backslash G_A$. Since $G_A$ admits a lattice, it is unimodular \cite{Milnor1976}*{Lemma 6.2}; hence $\tr A=0$. Every invariant generalized complex structure
on $\g_A$ descends to $X$. The inclusion of invariant forms induces maps
$GH^j(\g_A)\to GH^j(X)$ for each of the four cohomologies.

Averaging over $X$ shows that invariant forms are always a direct
summand of $U^\bullet(X)$, compatibly with $\partial$ and $\db$. Hence
all four maps are injective. Moreover, they are all isomorphisms as soon
as the map for $GH_{\db}$ is one
(Proposition~\ref{prop:reduction-to-dbar}), similarly to \cite{ACK}*{Corollary 5.2} and \cite{Angella2019}*{Corollary 1.3}.

To prove surjectivity, we use fibrations adapted to the structure: the construction was first presented by \cite{ACK} and discussed in the almost abelian case in~\cite{PaperI}.
Suppose that $\h=\h_c\oplus\mathfrak k$ with $A$-invariant summands, that
$\mathfrak b=\g_A/\mathfrak k$ carries a complex structure $J_B$, and
that a closed two-form $\omega$ on $\g_A$ restricts to a symplectic form
on $\mathfrak k$. Then $\rho=e^{\ii\omega}\wedge\pi^*\Omega$, with
$\Omega$ a generator of the canonical line of $J_B$ and $\pi: \mathfrak{g}_A \rightarrow \mathfrak{b}$ the natural projection, defines a
generalized complex structure of type $k=\dim_\C\mathfrak b$
\cite{PaperI}*{Theorem 3.14}. Model~II structures arise in this way; Model~I structures
do not. At the Lie algebra level, $\LL$ is then an extension
\[
0\longrightarrow\mathfrak S\longrightarrow\LL
\longrightarrow\LL_B\longrightarrow0
\]
of the eigenspace $\LL_B=\mathfrak b^{0,1}\oplus\mathfrak b^{*1,0}$ of
$J_B$ by an abelian ideal $\mathfrak S\cong\mathfrak k_\C$, compatibly
with the canonical characters (Lemma~\ref{lem:L-extension}). If $\Gamma$
is compatible with $\mathfrak k$, the solvmanifold $X$ is a symplectic
torus bundle
\[
\mathbb T\longrightarrow X\longrightarrow B
\]
over a compact complex almost abelian solvmanifold $B$. The Lie algebroid
spectral sequence of this fibration \cite{ACK} and the Hochschild--Serre 
spectral sequence \cite{HS53} of $\mathfrak S\subset\LL$ have second pages given, respectively, by
the Dolbeault cohomology of $B$ with coefficients in the flat bundles
$\mathcal H^q$ of fiberwise cohomology and by its invariant counterpart.
We prove (Theorem~\ref{thm:invariance} and
Corollary~\ref{cor:BC-invariance}) the following. If invariant forms
compute $H^b(B,\Omega^a_B\otimes\mathcal H^q)$ for all $a$, $b$ and $q$,
then
\[
GH^j(\g_A)\cong GH^j(X)
\]
for all $j$, for the generalized Dolbeault, Bott--Chern and Aeppli
cohomologies alike.

In the nilpotent case, comparisons of this kind were obtained in \cite{ACK} by using the nilpotency of the relevant coefficient modules. This approach builds on the classical Dolbeault and Bott--Chern comparisons for invariant complex structures on nilmanifolds \cites{Console2001, Angella2011}, as well as Kasuya's finite-dimensional models for the Dolbeault cohomology of certain solvmanifolds \cite{Kasuya13}.

For almost abelian solvmanifolds, the bundles $\mathcal H^q$ are in general not
nilpotent. They are, however, governed by the same combinatorics as the
invariant computation. The bundle $\mathcal H^q$ admits a filtration whose successive quotients are
holomorphic line bundles $\mathcal E_\kappa$. Each $\mathcal E_\kappa$ is
the trivial bundle with $\db$-operator $\db_B-\kappa\,(e^0)^{0,1}$, and
$\kappa$ ranges over the sums of $q$ eigenvalues of $A|_{\mathfrak k}$
(Lemma~\ref{lem:characters}). For Model~II this gives the following
(Corollary~\ref{cor:modelII-quotients}). Suppose that invariant forms
compute the Dolbeault cohomology of $B$ with coefficients in
$\mathcal E_\kappa$, for every sum $\kappa$ of a subfamily of
$(\zeta_1,-\zeta_1,\ldots,\zeta_s,-\zeta_s)$. Then
\[
\dim GH^{n-p}_{\db}(X)=\dim GH^{p-n}_{\db}(X)=N_p+N_{p-1},
\]
and invariant forms also compute the Bott--Chern and Aeppli cohomologies
of $X$.

The last hypothesis depends on the lattice, and not only on $\g_A$ and
$\mathcal J$, in line with what is known for complex structures on some solvmanifolds \cite{Kasuya13}. Indeed, consider $A=\operatorname{diag}(0,\zeta,-\zeta)$ with
$\zeta\ne0$. Let $X$ be the product of a circle of length $y_0$ with a
compact quotient of the three-dimensional solvable group $\mathrm{Sol}$,
endowed with its Model~II structure of type $1$. Invariant forms compute
$GH_{\db}(X)$ if and only if $\zeta y_0\notin2\pi\Z$. In that case
\[
\bigl(\dim GH^2_{\db},\ldots,\dim GH^{-2}_{\db}\bigr)(X)=(1,2,2,2,1).
\]
If instead $\zeta y_0\in2\pi\Z$, these dimensions are $(1,4,6,4,1)$
(Example~\ref{ex:lattice-dependence}). The comparison for Model~I
structures, whose natural fibration has complex fibers over a symplectic
base, is not treated here.

\medskip

   \subsection*{Acknowledgments}
    M. Andrade~P. is supported by FAPESP grant no. 2025/14730-8.

    L.~Camponês do Brasil M. is supported by FAPESP grants no. 2025/23958-2, 2026/09294-7. 
        
    L.~F.~Cavenaghi is supported by the Simons Foundation, grant SFI-MPS-T-Institutes-00007697, and the Ministry of Education and Science of the Republic of Bulgaria, grant DO1-239/10.12.2024.
    
    P.~A.~M.~Martins is supported by FAPESP grants no. 2024/07684-7, 2025/22312-1. 

   The authors are grateful to Lino Grama for suggesting the project and for his guidance and continued collaboration throughout its development.

    The authors thank Giovane Galindo, who kindly helped with fruitful insights during the conception of this manuscript.

The authors thank Beatrice Brienza, whose comments and questions led to Theorem \ref{thm:invariance}.

    Part of this work was conceived while P.~A.~M.~Martins, L.~Camponês do Brasil M., and  M. Andrade~P. were visiting the ICMS-Sofia under L.~Katzarkov's research group ``Theory of Atoms'', for which they thank the hospitality and leadership.
    \vspace{1em}

\section{Spinors and cohomology with canonical coefficients}
\label{sec:spinors}

In this section, we express the invariant generalized Dolbeault complex as a Lie algebra cochain complex. The Lie algebra is the complex eigenspace $\LL$ of the generalized complex structure, and the coefficient space is its canonical line $K$. We need to establish two points: graded differential forms must be identified with alternating $K$-valued maps on $\LL$, and their differential must agree with the differential determined by the bracket on $\LL$ and its action on $K$. 

\subsection{Lie algebra cohomology and its coefficient differential}
\label{sec:Lie-algebra-cohomology}
Let $\mathfrak a$ be a finite-dimensional complex Lie algebra.
A left module over $\mathfrak a$ is a complex vector space $M$
equipped with a bilinear action $(u,m)\mapsto u\cdot m$ satisfying
\[
[u,v]\cdot m=u\cdot(v\cdot m)-v\cdot(u\cdot m), \qquad u,v \in \mathfrak{a}, \, m \in M.
\]
The space of degree-$p$ cochains is
\[
C^p(\mathfrak a,M)
=\operatorname{Hom}_{\C}(\bigwedge^p\mathfrak a,M)
\cong\bigwedge^p\mathfrak a^*\otimes M.
\]
In degree zero, $C^0(\mathfrak a,M)=M$.
The Chevalley--Eilenberg differential
$\d_{\mathfrak a,M}:C^p(\mathfrak a,M)\to C^{p+1}(\mathfrak a,M)$ is
\begin{align}
(\d_{\mathfrak a,M}f)(v_0,\ldots,v_p)
&=\sum_{i=0}^p(-1)^i v_i\cdot
 f(v_0,\ldots,\widehat v_i,\ldots,v_p)\nonumber\\
&\quad+\sum_{0\le i<j\le p}(-1)^{i+j}
 f([v_i,v_j],v_0,\ldots,\widehat v_i,\ldots,
 \widehat v_j,\ldots,v_p).
\label{eq:CE}
\end{align}
In the first two degrees, the formula reads
\begin{equation}\label{eq:CE-low-degrees}
\begin{aligned}
(\d_{\mathfrak a,M}m)(u)&=u\cdot m,\\
(\d_{\mathfrak a,M}f)(u,v)
&=u\cdot f(v)-v\cdot f(u)-f([u,v]).
\end{aligned}
\end{equation}
For example, applying the second formula to the first gives
\[
(\d_{\mathfrak a,M}^2m)(u,v)
=u\cdot(v\cdot m)-v\cdot(u\cdot m)-[u,v]\cdot m=0.
\]
In higher degrees, the same representation identity cancels the
action terms, while the Jacobi identity cancels the remaining
double-bracket terms. Hence $\d_{\mathfrak a,M}^2=0$.
The resulting cohomology is
\[
H^p(\mathfrak a,M)
=\frac{\ker \left( \d_{\mathfrak a,M}:C^p(\mathfrak{a},M)  \to C^{p+1}(\mathfrak{a},M) \right)}
{\im \left(\d_{\mathfrak a,M}:C^{p-1}(\mathfrak{a},M) \to C^p(\mathfrak{a},M) \right)},
\]
where $C^{-1}(\mathfrak{a},M)=0$. In particular,
$H^0(\mathfrak a,M)=\{m\in M:u\cdot m=0\text{ for every }u\in\mathfrak a\}$.

Two coefficient choices will occur throughout the paper.
For the trivial module $M=\C$, the action sum in~\eqref{eq:CE}
vanishes. We write the resulting differential as $\d_{\mathfrak a}$.
It is the degree-one derivation of $\bigwedge^\bullet\mathfrak a^*$
determined by
\begin{equation}\label{eq:CE-trivial}
\d_{\mathfrak a}c=0\quad(c\in\C),\qquad
(\d_{\mathfrak a}\eta)(u,v)=-\eta([u,v])
\quad(\eta\in\mathfrak a^*).
\end{equation}
If $\mathfrak a=\g_\C$ is the complexification of the Lie algebra of
a real Lie group, this is exactly the exterior derivative restricted
to invariant forms: the coefficient functions of an invariant form
are constant, and the usual formula for its exterior derivative
reduces to the bracket sum in~\eqref{eq:CE}.

More generally, let $M$ be one-dimensional and choose a nonzero
vector $\varepsilon\in M$. Its action has the form
$u\cdot\varepsilon=\chi(u)\varepsilon$ for a linear functional
$\chi\in\mathfrak a^*$ satisfying $\chi([\mathfrak a,\mathfrak a])=0$. We call $\chi$ a \emph{character} and denote the module by $\C_\chi$.
Writing a cochain as $\alpha\otimes\varepsilon$, formula~\eqref{eq:CE}
becomes
\begin{equation}\label{eq:twist}
\d_{\mathfrak a,\chi}\alpha
=\d_{\mathfrak a}\alpha+\chi\wedge\alpha.
\end{equation}
Indeed, evaluation of $\chi\wedge\alpha$ is exactly the first sum in~\eqref{eq:CE}.

Next, we apply this construction with $\mathfrak a=\LL$ and $M=K$.
The remaining work is to find the action on $K$ and to prove that
its coefficient differential is the generalized Dolbeault operator.

\subsection{From graded forms to cochains on the eigenspace}

Let $\g$ be a real Lie algebra of dimension $2n$. All its forms
will be complexified, and from now on $\d=\d_{\g_\C}$ denotes the
trivial-coefficient differential just described. Set
\[
E=(\g\oplus\g^*)_\C,\qquad
S=\bigwedge^\bullet\g^*_\C,
\qquad
\langle X+\xi,Y+\eta\rangle
=\tfrac12\bigl(\xi(Y)+\eta(X)\bigr).
\]
The Clifford action of $E$ on $S$ is
\begin{equation}\label{eq:clifford}
c(X+\xi)\alpha=\iota_X\alpha+\xi\wedge\alpha,
\qquad c(u)c(v)+c(v)c(u)=2\langle u,v\rangle\id_S.
\end{equation}
With the convention $(\ad_X^*\eta)(Y)=-\eta([X,Y])$, the invariant
Courant bracket, following  \cite{Barberis}, takes the form
\begin{equation}\label{eq:cotangent-bracket}
\br{X+\xi}{Y+\eta}
=[X,Y]+\ad_X^*\eta-\ad_Y^*\xi.
\end{equation}
It is the semidirect-product Lie bracket on $E$.

Let $\mathcal J$ be a generalized almost complex structure: a real
orthogonal endomorphism of $\g\oplus\g^*$ satisfying
$\mathcal J^2=-\id$. Its $+\ii$ eigenspace $\LL$ and its conjugate
are complementary maximal isotropic subspaces of $E$.
The structure is integrable precisely when
$\br{\LL}{\LL}\subseteq\LL$, in which case $\LL$ is a complex
Lie algebra with this restricted bracket.

The canonical space associated with $\LL$ is
\[
K=\{\alpha\in S:c(v)\alpha=0\text{ for every }v\in\LL\}.
\]
The linear pure-spinor construction yields a nonzero element
$\rho\in K$ \cite{Gualtieri2011}*{Section~1.1}. It has the form
\[
\rho=e^{\beta+\ii\omega}\wedge\Omega,
\qquad \Omega=\theta_1\wedge\cdots\wedge\theta_k,
\qquad
\omega^{n-k}\wedge\Omega\wedge\overline\Omega\ne0,
\]
for real two-forms $\beta,\omega$; the integer $k$ is its type.
The next lemma shows that $K=\C\rho$ and explains why this line
appears as a coefficient space. Starting from $\rho$, Clifford
multiplication by $\overline\LL$ generates every form. The pairing
then identifies each such factor with a covector on $\LL$.

\begin{lemma}\label{lem:grading}
Choose a basis $u_1,\ldots,u_{2n}$ of $\LL$ and the basis
$w^1,\ldots,w^{2n}$ of $\overline\LL$ determined by
$2\langle u_a,w^b\rangle=\delta_a^b$.
For $0\ne\rho\in K$, the vectors
\[
s_I=c(w^{i_1})\cdots c(w^{i_r})\rho,
\qquad I=\{i_1<\cdots<i_r\}\subseteq\{1,\ldots,2n\},
\]
with $s_\varnothing=\rho$, form a basis of $S$, and $K=\C\rho$.
Define
\begin{equation}\label{eq:grading}
\UU^{n-r}=\bigwedge^r\overline\LL\cdot K
\qquad(0\le r\le2n).
\end{equation}
Then
\[
S=\bigoplus_{j=-n}^n\UU^j,\qquad
\dim\UU^j=\binom{2n}{n-j},\qquad
\overline{\UU^j}=\UU^{-j},
\]
and Clifford multiplication satisfies
\[
c(\LL)\UU^j\subseteq\UU^{j+1},\qquad
c(\overline\LL)\UU^j\subseteq\UU^{j-1}.
\]
Moreover, there are natural vector-space isomorphisms
\begin{equation}\label{eq:spinor-cochain-map}
\begin{aligned}
F_r:\UU^{n-r}&\longrightarrow C^r(\LL,K)
=\bigwedge^r\LL^*\otimes K,\\
F_r(\alpha)(v_1,\ldots,v_r)&=c(v_r)\cdots c(v_1)\alpha.
\end{aligned}
\end{equation}
We set $\UU^j=0$ outside $-n\le j\le n$ and extend $F_r$ by zero
outside $0\le r\le2n$.
\end{lemma}

\begin{proof}
Put $a_a=c(u_a)$ and $b_a=c(w^a)$. Isotropy and the Clifford
relation imply
\[
a_aa_b+a_ba_a=b_ab_b+b_bb_a=0,
\qquad a_ab_b+b_ba_a=\delta_a^b\id_S.
\]
Moving $a_a$ through a monomial and using $a_a\rho=0$ gives
\begin{equation}\label{eq:contraction}
a_as_I=
\begin{cases}
0,&a\notin I,\\
(-1)^{t-1}s_{I\setminus\{a\}},&a=i_t.
\end{cases}
\end{equation}
Suppose a linear relation among the $s_I$ has a nonzero coefficient
indexed by a set $I=\{i_1<\cdots<i_r\}$ of maximal size.
Applying $a_{i_r}\cdots a_{i_1}$ kills every other term with a
nonzero coefficient and sends $s_I$ to $\rho$, a contradiction.
The $2^{2n}$ monomials therefore form a basis of $S$.
For fixed $a$, distinct monomials containing $a$ are sent by $a_a$
to distinct basis vectors, up to sign. A vector annihilated by all
$a_a$ can consequently have only its $s_\varnothing$ coefficient
nonzero. This proves $K=\C\rho$.
Grouping monomials by their length proves the direct decomposition,
the dimensions, and the two grading shifts.

Conjugation can also be read directly from this basis. The nonzero
top monomial $s_{\rm top}=b_1\cdots b_{2n}\rho$ is annihilated by
all $b_a$. Conjugating the definition of $K$ identifies this common
annihilator with $\overline K=\C\bar\rho$, so $s_{\rm top}$ spans it.
Repeated use of~\eqref{eq:contraction} shows that
$\bigwedge^r\LL\cdot s_{\rm top}$ is spanned by all $s_I$ of
length $2n-r$. Hence
\[
\overline{\UU^{n-r}}
=\bigwedge^r\LL\cdot\overline K
=\UU^{r-n}.
\]

Finally, the values of $F_r$ lie in $\UU^n=K$ by the grading shifts,
and the map is alternating because Clifford operators from $\LL$
anticommute. For $w_1,\ldots,w_r\in\overline\LL$, the relations give
\[
F_r\bigl(c(w_1)\cdots c(w_r)\rho\bigr)(v_1,\ldots,v_r)
=\det\bigl(2\langle v_i,w_j\rangle\bigr)\rho.
\]
This identifies the monomial basis with the usual alternating
cochain basis and proves the last assertion.
\end{proof}

At this point $F_r$ is only a vector-space identification: the
notation $C^r(\LL, K)$ denotes alternating maps, without yet specifying
an action on $K$. The index shift is already explained, however.
The canonical line occupies generalized degree $n$, and each
$\overline\LL$ factor lowers that degree by one while adding one
cochain argument. Thus generalized degree $n-r$ corresponds to
cochain degree $r$. To identify complexes, we must now examine $\d$.

\subsection{Integrability and the two components of the exterior differential}

The differential $\d$ has ordinary form degree one. Its effect on
the generalized grading requires a separate calculation, because a
space $\UU^j$ usually contains forms of several ordinary degrees.
For a basis $X_a$ of $\g$ and dual basis $\xi^a$, one has
\begin{equation}\label{eq:dcubic}
\d=-\sum_{a<b}\xi^a\wedge\xi^b\wedge\iota_{[X_a,X_b]}.
\end{equation}
Both sides are degree-one derivations of the exterior algebra, and
they agree on scalars and one-forms by~\eqref{eq:CE-trivial}, which
proves the formula. Each summand is a product of three Clifford
operators; hence, the derivation can be rewritten in terms of the Clifford action as 
\begin{equation}\label{eq:dec_dif_Clifact}
    \d=-\sum_{a<b} c(\xi^a)c(\xi^b)c([X_a,X_b]).
\end{equation}

Splitting each factor between $\LL$ and $\overline\LL$
and applying Lemma~\ref{lem:grading}, we see that the only possible
generalized degree shifts are $3,1,-1,-3$.

Let $\operatorname{pr}_j:S\to\UU^j$ be the projections of the
direct decomposition. For $\alpha\in\UU^j$, define
\begin{equation}\label{eq:projected-differentials}
\begin{aligned}
\partial\alpha&=\operatorname{pr}_{j+1}(\d\alpha),&
\db\alpha&=\operatorname{pr}_{j-1}(\d\alpha),\\
\delta_+\alpha&=\operatorname{pr}_{j+3}(\d\alpha),&
\delta_-\alpha&=\operatorname{pr}_{j-3}(\d\alpha).
\end{aligned}
\end{equation}
Thus these operators are the components of $\d$ with respect to the
generalized grading, and
$\d=\delta_++\partial+\db+\delta_-$.
Integrability will eliminate $\delta_\pm$. To see how the bracket
controls these components, we first record an operator identity that
will also determine the coefficient action on $K$.

For $v=X+\xi\in E$, define
\begin{equation}\label{eq:Cartan-operator}
D_v=\d c(v)+c(v)\d
=\mathcal L_X+(\d\xi)\wedge,
\qquad \mathcal L_X=\d\iota_X+\iota_X\d.
\end{equation}
For a vector $v=X$, this is Cartan's formula for Lie differentiation
on invariant forms. For a covector $v=\xi$, it is multiplication by
the two-form $\d\xi$: the terms involving $\xi\wedge\d$ cancel by
the Leibniz rule. The operator $D_v$ will express a Lie algebra
action; it is not a cochain differential.

Using ordinary operator commutators $[P,Q]=PQ-QP$, we have
\begin{equation}\label{eq:D-identities}
[D_u,c(v)]=c(\br u v),\qquad [D_u,D_v]=D_{\br u v}.
\end{equation}
Indeed, for $u=X+\xi$ and $v=Y+\eta$, Cartan's identities give
\[
\begin{aligned}
\bigl[D_u,c(v)\bigr]
&=\iota_{[X,Y]}+
 (\mathcal L_X\eta-\iota_Y\d\xi)\wedge\\
&=\iota_{[X,Y]}+
 (\mathcal L_X\eta-\mathcal L_Y\xi)\wedge,
\end{aligned}
\]
where $\iota_Y\d\xi=\mathcal L_Y\xi$ because $\xi(Y)$ is constant
on invariant sections. Since $\mathcal L_X\eta=\ad_X^*\eta$,
this is $c(\br u v)$. Similarly,
\[
\begin{aligned}
\bigl[D_u,D_v\bigr]
&=\mathcal L_{[X,Y]}+
 (\mathcal L_X\d\eta-\mathcal L_Y\d\xi)\wedge\\
&=\mathcal L_{[X,Y]}+
 \d(\mathcal L_X\eta-\mathcal L_Y\xi)\wedge
=D_{\br u v}.
\end{aligned}
\]

\begin{proposition}\label{prop:integrable-differentials}
If $\mathcal J$ is integrable, then $\delta_+=\delta_-=0$. Consequently,
\[
\d=\partial+\db,\qquad
\partial^2=\db^2=0,\qquad
\partial\db+\db\partial=0.
\]
Moreover, a non-degenerate pure spinor $\rho$ defines an integrable
structure if and only if $\d\rho=c(z)\rho$ for some $z\in \overline{\mathfrak{L}}$.
\end{proposition}

\begin{proof}
    The shift $-3$ in~\eqref{eq:dcubic} comes from the Clifford products of three
    factors $\xi^a,\xi^b,[X_a,X_b] \in \overline\LL$ in the expression~(\ref{eq:dec_dif_Clifact}). Let $\alpha \in \mathfrak{U}^{j}$, the degree $-3$ part in $\d \alpha$ is $c(v)c(u)c(w) \d \alpha$ for $v,u,w \in \mathfrak{L}$, disregarding some non-zero scalar. The operator identities just proved imply
    \[
    \begin{aligned}\label{eq:integrable_delta=0}
    c(w)c(v)c(u)\d\alpha
    &=c(w)c(v)D_u\alpha\\
    &=c(w)\bigl(D_uc(v)-c(\br u v)\bigr)\alpha\\
    &=-c(w)c(\br u v)\alpha
    =-2\langle\br u v,w\rangle\alpha.
    \end{aligned}
    \]
    If $\LL$ is involutive, this tensor vanishes by isotropy. Hence
    $\delta_-=0$ on all of $S$. The reality of $\d$ and the conjugation statement in Lemma~\ref{lem:grading} give $\delta_+=0$ as well.
    Now the components of $\d^2$ with shifts $2,0,-2$ vanish separately,
    giving the three asserted differential identities.

    The relation~\ref{eq:integrable_delta=0} holds for $\alpha = \rho \in \mathfrak{U}^{n}$. For integrable $\mathcal J$, the equality
    $\d\rho=\db\rho\in\UU^{n-1}=c(\overline\LL)\rho$
    provides $z$. Conversely, if $\d\rho=c(z)\rho$, its triple
    contraction by elements of $\LL$ is zero. The preceding calculation
    then gives $\langle\br u v,w\rangle=0$ for all $u,v,w\in\LL$.
    Maximal isotropy implies $\br{\LL}{\LL}\subseteq\LL$, proving
    integrability.
\end{proof}

We now assume integrability. The sequence
\[
0\longrightarrow\UU^n\xrightarrow{\db}\UU^{n-1}
\xrightarrow{\db}\cdots\xrightarrow{\db}\UU^{-n}
\longrightarrow0
\]
is a cochain complex when $\UU^{n-p}$ is assigned cochain degree
$p$. Its cohomology, recorded in generalized degree $j$, is
\begin{equation}\label{eq:generalized-Dolbeault-definition}
GH_{\db}^j(\g)=
\frac{\ker(\db:\UU^j\to\UU^{j-1})}
{\im(\db:\UU^{j+1}\to\UU^j)}.
\end{equation}
We have now defined the complex we want to compute. The maps $F_p$
identify its spaces with $C^p(\LL,K)$; the next step identifies
the action on $K$ that makes them maps of complexes.

\subsection{The canonical action and the identification of complexes}

Clifford multiplication by $\LL$ annihilates $K$, so it does not
provide the coefficient action we need. That action is instead
obtained from the operators $D_v$ in~\eqref{eq:Cartan-operator}.
For $v\in\LL$ and $\rho\in K=\UU^n$,
\[
D_v\rho=c(v)\d\rho=c(v)\db\rho\in\UU^n=K.
\]
Thus $D_v$ preserves $K$, and the second identity
in~\eqref{eq:D-identities} says exactly that its restriction is a
representation of $\LL$. Since $K$ is one-dimensional, there is a
character $\chi\in\LL^*$ determined by
\begin{equation}\label{eq:character}
D_v\rho=c(v)\d\rho=\chi(v)\rho\qquad(v\in\LL).
\end{equation}
It is indeed a character, since
\[
\chi(\br u v)\rho=[D_u,D_v]\rho
=(\chi(u)\chi(v)-\chi(v)\chi(u))\rho=0.
\]
Replacing $\rho$ by a nonzero constant multiple does not change
$\chi$. We therefore obtain a well-defined canonical module
$K\cong\C_\chi$.

This construction already explains the coefficient term in degree
zero. By the definition of $F_1$,
\[
F_1(\db\rho)(v)=c(v)\d\rho=\chi(v)\rho,
\]
which is precisely $\d_{\LL,K}\rho$ as given by
\eqref{eq:CE-low-degrees}. The following proposition proves the
same compatibility in every degree.

\begin{proposition}\label{prop:canonical-complex}
With the action $v\cdot s=D_vs$ on $K$, the maps
$F_p$ of~\eqref{eq:spinor-cochain-map} make the following diagram commute
\[
\begin{tikzcd}[row sep=large, column sep=large]
\cdots \arrow[r] & \mathfrak{U}^{n-p+1} \arrow[r, "\overline{\partial}"] \arrow[d, "F_{p-1}"] & \mathfrak{U}^{n-p} \arrow[r, "\overline{\partial}"] \arrow[d, "F_{p}"] & \mathfrak{U}^{n-p-1} \arrow[r] \arrow[d, "F_{p+1}"] & \cdots \\
\cdots \arrow[r] & C^{p-1}(\LL,K)  \arrow[r, "\d_{\LL,K}"] & C^{p}(\LL,K) \arrow[r, "\d_{\LL,K}"] & C^{p+1}(\LL,K) \arrow[r] & \cdots
\end{tikzcd}
\]
in other words,
\begin{equation}\label{eq:chain-identification}
F_{p+1}\db=\d_{\LL,K}F_p.
\end{equation}
After choosing $\rho$ to identify $K$ with $\C_\chi$, the differential
on the right is $\d_{\LL,\chi}=\d_\LL+\chi\wedge$.
\end{proposition}

\begin{proof}
Take $\alpha\in\UU^{n-p}$ and $v_0,\ldots,v_p\in\LL$, and
write $c_i=c(v_i)$ and $D_i=D_{v_i}$.
The grading gives
$c_p\cdots c_0\alpha=c_p\cdots c_0\partial\alpha=0$.
Therefore
\[
F_{p+1}(\db\alpha)(v_0,\ldots,v_p)
=c_p\cdots c_0\d\alpha.
\]
Successively substitute $c_i\d=D_i-\d c_i$ to obtain
\[
c_p\cdots c_0\d\alpha
=\sum_{i=0}^p(-1)^i
c_p\cdots c_{i+1}D_i c_{i-1}\cdots c_0\alpha;
\]
the last term, $(-1)^{p+1}\d(c_p\cdots c_0\alpha)$, is zero.
Move $D_i$ to the far left using
$c_jD_i=D_ic_j-c(\br{v_i}{v_j})$ for $j>i$.
The term in which $D_i$ reaches the left acts on a value of $F_p$
in $K$, so it gives the $i$th coefficient-action term of
\eqref{eq:CE}. A bracket contribution for $i<j$ initially has
sign $(-1)^{i+1}$. To place that bracket in the first cochain
argument, move its Clifford factor to the far right, across
$j-1$ other factors. Because all these factors belong to the
isotropic space $\LL$, each interchange contributes only a minus
sign. The resulting sign is $(-1)^{i+j}$.
Writing $\widetilde\alpha=F_p(\alpha)$, we obtain
\begin{align*}
F_{p+1}(\db\alpha)(v_0,\ldots,v_p)
&=\sum_{i=0}^p(-1)^i v_i\cdot
\widetilde\alpha(v_0,\ldots,\widehat v_i,\ldots,v_p)\\
&\quad+\sum_{i<j}(-1)^{i+j}
\widetilde\alpha(\br{v_i}{v_j},v_0,\ldots,
\widehat v_i,\ldots,\widehat v_j,\ldots,v_p).
\end{align*}
This is the coefficient differential~\eqref{eq:CE}, proving
\eqref{eq:chain-identification}.
\end{proof}

\begin{corollary}
    Set $\bar\chi(\bar v)=\overline{\chi(v)}$ and let $GH_\partial$ denote the cohomology of $\partial$
   \begin{equation}\label{eq:ACK}
    GH_{\db}^{n-p}(\g)\cong H^p(\LL,\C_\chi),\qquad
    GH_{\partial}^{-n+p}(\g)
    \cong H^p(\overline\LL,\C_{\bar\chi}),
    \end{equation}
    If $\d\rho=c(z)\rho$, then
    \begin{equation}\label{eq:chi}
    \chi(v)=2\langle z,v\rangle\qquad(v\in\LL).
    \end{equation}
    In particular, $\chi=0$ if and only if the invariant generator
    $\rho$ is closed. 
\end{corollary}
\begin{proof}
    The first isomorphism in cohomology follows from the isomorphism between the complexes in the proposition above. Since $\d$ is real and conjugation reverses the grading, it exchanges $\partial$ and $\db$. Conjugating the construction gives the second isomorphism in~\eqref{eq:ACK}.

    For the character formula, $c(v)\rho=0$ and the Clifford relation give
    \[
    \chi(v)\rho=c(v)c(z)\rho
    =2\langle v,z\rangle\rho-c(z)c(v)\rho
    =2\langle v,z\rangle\rho.
    \]
    This value is independent of the choice of $z$: any two choices
    differ by an element of the Clifford annihilator $\Ann(\rho)=\{u\in E:c(u)\rho=0\}=\LL$, which pairs trivially
    with $\LL$.
    Finally, $F_1(\d\rho)=\chi\otimes\rho$, and $F_1$ is an
    isomorphism. Thus $\chi=0$ is equivalent to $\d\rho=0$.
\end{proof}

\subsection{Invariant cohomology and the complex on a compact quotient}

So far, our cohomology groups have been defined using invariant
forms. We will eventually compare them with the cohomology of all
smooth forms on a compact quotient. To formulate that comparison,
we first record the other cohomologies associated with
$(\UU^\bullet,\partial,\db)$, and then describe explicitly what
changes in the cochain complex when coefficients are allowed to
be smooth functions.

Besides $GH_{\db}^j(\g)$, we consider
\begin{align*}
GH_\partial^j(\g)
&=\frac{\ker(\partial:\UU^j\to\UU^{j+1})}
{\im(\partial:\UU^{j-1}\to\UU^j)},\\
GH_{BC}^j(\g)
&=\frac{\ker(\partial|_{\UU^j})\cap\ker(\db|_{\UU^j})}
{\partial\db(\UU^j)},\\
GH_A^j(\g)
&=\frac{\ker(\partial\db:\UU^j\to\UU^j)}
{\partial(\UU^{j-1})+\db(\UU^{j+1})}.
\end{align*}
Bott--Chern classes are represented by forms closed under both
operators, modulo forms $\partial\db\beta$.
Aeppli classes are represented by forms annihilated by
$\partial\db$, modulo sums of a $\partial$-exact form and a
$\db$-exact form. 

Now let $X=\Gamma\backslash G$ be a compact quotient endowed with
the invariant generalized complex structure under consideration,
and let $\g$ be the Lie algebra of $G$.
The spaces $\LL$, $K$, and $\UU^j$ extend to invariant subbundles
of the complexified generalized tangent bundle and exterior-form
bundle on $X$. Denote the first two by $L_X$ and $K_X$, and write
$U^j(X)$ for the smooth sections of the $j$th graded subbundle.
Replacing $\UU^j$ by $U^j(X)$ in the preceding definitions gives
$GH_\partial^j(X)$, $GH_{BC}^j(X)$, and $GH_A^j(X)$, as well as
$GH_{\db}^j(X)$. The inclusion of invariant forms commutes with
both differentials and therefore induces comparison maps
\[
GH_\star^j(\g)\longrightarrow GH_\star^j(X),
\qquad \star\in\{\partial,\db,BC,A\}.
\]
The question is whether every cohomology class on the right can
be represented invariantly and whether an invariant class that
vanishes there already vanishes in the invariant complex.

We can describe the global Dolbeault complex using the Lie algebra
cohomology construction introduced above. Choose an invariant
generator $\rho$ of the canonical line. It descends to a nowhere
vanishing section of $K_X$, so every smooth section of $K_X$ is
uniquely $f\rho$ for some $f\in C^\infty(X,\C)$. Set
\[
\mathcal M_\rho=C^\infty(X,\C)\rho.
\]
This is now an infinite-dimensional coefficient space. For
$v=Y+\xi\in\LL$, let $v^L$ be the corresponding invariant section
of $L_X$ and let $Y^L$ be its vector part. Define
\begin{equation}\label{eq:global-canonical-module}
v\cdot(f\rho)=\bigl(Y^L(f)+f\chi(v) \bigr)\rho.
\end{equation}
The term $Y^L(f)$ differentiates the function multiplying $\rho$;
the term $f\chi(v)$ is the previously established action on the
canonical line. For constant $f$, this is exactly the module
$K\cong\C_\chi$ used in the invariant calculation.

Equation~\eqref{eq:global-canonical-module} defines a representation
of $\LL$. Indeed, the character values are constant,
$\chi(\br u v)=0$, and vector projection preserves brackets.
Writing $Y_u^L, Y_v^L$ for the vector parts, we therefore have
\[
u\cdot\bigl(v\cdot(f\rho)\bigr)
-v\cdot\bigl(u\cdot(f\rho)\bigr)
=[Y_u^L,Y_v^L](f)\rho
=\br u v\cdot(f\rho).
\]
We can consequently form the Chevalley--Eilenberg complex
$C^\bullet(\LL,\mathcal M_\rho)$.
Since $\LL$ is finite-dimensional, its cochains identify with
\[
C^p(\LL,\mathcal M_\rho)
\cong C^\infty\bigl(X,\bigwedge^p\LL^*\bigr)\rho
\cong\Gamma\bigl(X,\bigwedge^pL_X^*\otimes K_X\bigr).
\]
Thus a cochain can be written as $\eta\rho$, where evaluating
$\eta$ on fixed $v_1,\ldots,v_p\in\LL$ produces a smooth function
on $X$.

In this notation, the coefficient differential is completely
explicit. Write
$\d_{\LL,\mathcal M_\rho}(\eta\rho)=(\d_X\eta)\rho$.
Formula~\eqref{eq:CE} and~\eqref{eq:global-canonical-module} give
\begin{align}
(\d_X\eta)(v_0,\ldots,v_p)
&=\sum_{i=0}^p(-1)^i
\Bigl(
Y_i^L\bigl(\eta(v_0,\ldots,\widehat v_i,\ldots,v_p)\bigr)
\nonumber\\[-2pt]
&\hspace{43mm}
+\chi(v_i)\eta(v_0,\ldots,\widehat v_i,\ldots,v_p)
\Bigr)\nonumber\\
&\quad+\sum_{i<j}(-1)^{i+j}
\eta(\br{v_i}{v_j},v_0,\ldots,\widehat v_i,\ldots,
\widehat v_j,\ldots,v_p),
\label{eq:global-cochain-differential}
\end{align}
where $Y_i^L$ is the vector part of $v_i^L$.
The last sum uses the bracket, just as in the invariant complex.
The terms containing $\chi$ retain the canonical coefficient
module. The derivatives $Y_i^L$ account for the variation of the
cochain over $X$. For example, in degree zero this formula reduces to
\[
(\d_X f)(v)=Y^L(f)+\chi(v)f.
\]
When all coefficients of $\eta$ are constant, the derivatives
vanish and~\eqref{eq:global-cochain-differential} reduces to
$\d_{\LL,\chi}$. The invariant inclusion is therefore the
inclusion of constant-coefficient cochains into this complex.

For compatibility with the bundle spectral sequence used later,
we also describe this same differential intrinsically.
The bundle $L_X$ is a Lie algebroid: its sections have the
restricted Courant bracket, and its vector projection
\[
a:L_X\longrightarrow T_\C X,\qquad a(Y+\xi)=Y,
\]
satisfies
\[
\br u{fv}=f\br u v+a(u)(f)v,
\qquad a(\br u v)=[a(u),a(v)].
\]
The map $a$ is called the anchor. Concretely, $a(v)$ is the vector
field that differentiates coefficient functions when the section
$v$ occurs in the differential.

The canonical line has an $L_X$-connection given by
\[
\nabla_v^K s=c(v)\d s,
\qquad v\in\Gamma(L_X),\quad s\in\Gamma(K_X).
\]
Its values lie in $K_X$: integrability gives
$\d s\in U^{n-1}(X)$, and Clifford multiplication by $L_X$ raises
the generalized degree to $n$.
By an $L_X$-connection we mean that the operator is linear over
smooth functions in $v$ and obeys
\[
\nabla_v^K(fs)=a(v)(f)s+f\nabla_v^K s.
\]
For the invariant frame, expansion of $\d(f\rho)$ gives
\[
\begin{aligned}
\nabla_{v^L}^K(f\rho)
&=c(v^L)(\d f\wedge\rho+f\d\rho)\\
&=\bigl(Y^L(f)+\chi(v)f\bigr)\rho,
\end{aligned}
\]
using $c(v^L)\rho=0$. This recovers precisely
\eqref{eq:global-canonical-module}. The representation identity
proved there implies flatness,
$[\nabla_u^K,\nabla_v^K]=\nabla_{\br u v}^K$:
it suffices to check this in the invariant frame, because curvature
is linear over smooth functions in both arguments.
The differential on bundle-valued cochains is formula~\eqref{eq:CE}
with $v_i\cdot$ interpreted as $\nabla_{v_i}^K$.
Its expression in the invariant frame is
\eqref{eq:global-cochain-differential}. Its cohomology is the
Lie algebroid cohomology $H^p(L_X, K_X)$.

Finally, the maps $F_p$ extend to smooth forms by
\[
\widetilde F_p(\alpha)(v_1,\ldots,v_p)
=c(v_p^L)\cdots c(v_1^L)\alpha,
\qquad \alpha\in U^{n-p}(X).
\]
They identify the generalized Dolbeault differential with the
coefficient differential just described
\cite{Gualtieri2011}*{Section~3.2}.
To see explicitly what the extension adds, take an invariant
$\alpha\in\UU^{n-p}$ and $f\in C^\infty(X,\C)$.
The product rule for $\d(f\alpha)$ and the Clifford relation give
\[
\widetilde F_{p+1}\bigl(\db(f\alpha)\bigr)
=f\,\d_{\LL,K}F_p(\alpha)
+(a^*\d f)\wedge F_p(\alpha), 
\]
where $(a^*\d f)(v)=Y^L(f)$ for $v=Y+\xi\in\LL$.
The first summand is the invariant calculation already proved;
the second gives exactly the derivatives in
\eqref{eq:global-cochain-differential}. Since every smooth form is
a finite sum of smooth functions times invariant forms, this
verifies the identification in every degree. Consequently,
\begin{equation}\label{eq:global-ACK}
GH_{\db}^{n-p}(X)
\cong H^p(\LL,\mathcal M_\rho)
\cong H^p(L_X,K_X).
\end{equation}

Thus the invariant and global calculations use the same bracket
and canonical character. Their coefficient spaces are respectively
$\C\rho$ and $C^\infty(X,\C)\rho$, and the latter carries the
additional differentiation shown in
\eqref{eq:global-canonical-module}. Determining when the inclusion
between these two complexes induces a cohomology isomorphism is
the comparison problem addressed in Section~\ref{sec:quotients}.

\vspace{1em}

\section{The almost abelian cochain complex}\label{sec:algebra}

In this section, $\g=\g_A=\R e_0\ltimes_A\h$ is an almost abelian Lie
algebra: $\dim\h=2n-1$, $[e_0,Y]=AY$ for $Y\in\h$, and $[\h,\h]=0$.
Let $e^0(e_0)=1$ and $e^0|_\h=0$. We show that the eigenspace $\LL$ of
every invariant generalized complex structure on $\g_A$ is again almost
abelian, and that its canonical character is a multiple of a single
homomorphism $\phi$. The complex $C^\bullet(\LL,\C_\chi)$ of
Proposition~\ref{prop:canonical-complex} then splits into two-term
complexes.

Choose a basis $e_1,\ldots,e_{2n-1}$ of $\h$ and complete $e^0$ to the
dual basis $e^0,\ldots,e^{2n-1}$. The only nonzero brackets of basis
vectors are $[e_0,e_j]=Ae_j$, so~\eqref{eq:dcubic} becomes
\begin{equation}\label{eq:dgA}
\d=-e^0\wedge\sum_{j=1}^{2n-1}e^j\wedge\iota_{Ae_j}.
\end{equation}
In particular, $\d S\subseteq e^0\wedge S$.

\

\begin{proposition}\label{thm:L-almost-abelian}
Let $\mathcal J$ be an invariant generalized complex structure on
$\g_A$, with eigenspace $\LL$ and canonical character $\chi$. Then
the map
\[
\phi:\LL\longrightarrow\C,\qquad \phi(X+\xi)=e^0(X),
\]
is a surjective Lie algebra homomorphism with abelian kernel
$\mathfrak n$. Hence, for any $Z\in\LL$ with $\phi(Z)=1$,
\[
\LL=\C Z\ltimes_C\mathfrak n,\qquad
C=\ad_Z|_{\mathfrak n}:\mathfrak n\to\mathfrak n,\qquad
\dim_\C\mathfrak n=2n-1.
\]
The endomorphism $C$ does not depend on the choice of $Z$.
Moreover, $\chi|_{\mathfrak n}=0$. Consequently,
\begin{equation}\label{eq:chi-phi}
\chi=\tau\phi,\qquad \tau=\chi(Z),
\end{equation}
and $\tau$ does not depend on the choice of $Z$.
\end{proposition}

\begin{proof}
The vector part of $\br uv$ lies in $[\g_A,\g_A]\subseteq\h_\C$
by~\eqref{eq:cotangent-bracket}. Hence $\phi(\br uv)=0$, and $\phi$ is a
homomorphism to the abelian Lie algebra $\C$. If $\phi=0$, the vector
projections of both $\LL$ and $\overline\LL$ would lie in $\h_\C$. This
contradicts $E=\LL\oplus\overline\LL$. Thus $\phi$ is surjective, and
$\mathfrak n=\ker\phi$ is an ideal of dimension $2n-1$.

Take $X+\xi,\,Y+\eta\in\mathfrak n$, so that $X,Y\in\h_\C$. Then
$[X,Y]=0$. For $W\in\h$ the convention
$(\ad_X^*\eta)(W)=-\eta([X,W])$ gives zero, while for $W=e_0$ it gives
$\eta(AX)$. Therefore
\[
\br{X+\xi}{Y+\eta}=\bigl(\eta(AX)-\xi(AY)\bigr)e^0 .
\]
By integrability, this element lies in $\LL$. Its pairing with $Z$ is
$\tfrac12\bigl(\eta(AX)-\xi(AY)\bigr)\phi(Z)$. Isotropy of $\LL$ forces
this pairing, and hence the bracket, to vanish. Thus $\mathfrak n$ is
abelian, and the semidirect-product decomposition follows. If
$\phi(Z')=1$, then $Z'-Z\in\mathfrak n$. Hence
$\ad_{Z'}|_{\mathfrak n}=\ad_Z|_{\mathfrak n}$.

For the character, fix $0\ne\rho\in K$. By~\eqref{eq:dgA} we can write
$\d\rho=e^0\wedge\beta$ for some $\beta\in S$. Let
$v=X+\xi\in\mathfrak n$, so that $e^0(X)=0$. Then
\[
c(v)(e^0\wedge\beta)
=\iota_X(e^0\wedge\beta)+\xi\wedge e^0\wedge\beta
=-e^0\wedge c(v)\beta .
\]
By~\eqref{eq:character}, $\chi(v)\rho=c(v)\d\rho\in e^0\wedge S$.
Suppose $\chi(v)\ne0$. Then $\rho\in e^0\wedge S$, so
$c(e^0)\rho=e^0\wedge\rho=0$, and $e^0\in\Ann(\rho)=\LL$. Since $e^0$ is
real, it would also lie in $\overline\LL$. This is impossible because
$\LL\cap\overline\LL=0$. Hence $\chi|_{\mathfrak n}=0$. The linear forms
$\chi$ and $\chi(Z)\phi$ agree on $Z$ and on $\mathfrak n$, which
gives~\eqref{eq:chi-phi}. Finally, $\chi(Z')=\chi(Z)$ whenever
$Z'-Z\in\mathfrak n$.
\end{proof}

\

From now on, fix $Z\in\LL$ with $\phi(Z)=1$. The dual space splits as
\[
\LL^*=\C\phi\oplus\mathfrak n^*,\qquad
\mathfrak n^*:=\{\xi\in\LL^*:\xi(Z)=0\}.
\]
Restriction to $\mathfrak n$ identifies $\mathfrak n^*$ with the dual of
$\mathfrak n$. Accordingly, for every $p$,
\begin{equation}\label{eq:cochain-splitting}
\bigwedge^p\LL^*
=\bigwedge^p\mathfrak n^*\oplus
\phi\wedge\bigwedge^{p-1}\mathfrak n^*,
\end{equation}
with $\bigwedge^p\mathfrak n^*=0$ outside $0\le p\le2n-1$.

\begin{definition}\label{def:PsiC}
Let $\Psi_C$ be the degree-zero derivation of $\bigwedge^\bullet\LL^*$
determined by
\[
\Psi_C(1)=0,\qquad
\Psi_C\xi=\xi\circ\ad_Z\quad(\xi\in\LL^*).
\]
With the coadjoint convention of~\eqref{eq:cotangent-bracket}, this
means $\Psi_C\xi=-\ad_Z^*\xi$ on $\LL^*$.
\end{definition}

The definition has two immediate consequences. First,
$\Psi_C\phi=\phi\circ\ad_Z=0$, because $\phi$ vanishes on brackets.
Second, for $\eta\in\mathfrak n^*$ one has $(\Psi_C\eta)(Z)=0$ and
$(\Psi_C\eta)|_{\mathfrak n}=\eta\circ C$. Thus $\Psi_C$ preserves
$\bigwedge^\bullet\mathfrak n^*$. On this subalgebra it is the derivation
extending $C^T\eta=\eta\circ C$:
\begin{equation}\label{eq:Psi-expanded}
\Psi_C(\eta_1\wedge\cdots\wedge\eta_p)
=\sum_{j=1}^p\eta_1\wedge\cdots\wedge C^T\eta_j
\wedge\cdots\wedge\eta_p
\qquad(\eta_j\in\mathfrak n^*).
\end{equation}
It also satisfies $\Psi_C(\phi\wedge\beta)=\phi\wedge\Psi_C\beta$.

\begin{lemma}
For every $\alpha\in\bigwedge^\bullet\LL^*$,
\begin{equation}\label{eq:almost-d}
\d_\LL\alpha=-\phi\wedge\Psi_C\alpha.
\end{equation}
\end{lemma}

\begin{proof}
Both sides are degree-one derivations of $\bigwedge^\bullet\LL^*$.
For the right-hand side, this holds because $\Psi_C$ is a degree-zero
derivation and $\phi$ has degree one. It is therefore enough to compare
the two sides on scalars and on $\LL^*=\C\phi\oplus\mathfrak n^*$.
Both vanish on scalars. By~\eqref{eq:CE-trivial},
$(\d_\LL\phi)(u,v)=-\phi(\br uv)=0$, and also $\Psi_C\phi=0$.

Let $\eta\in\mathfrak n^*$. Since $\mathfrak n$ is abelian,
$(\d_\LL\eta)(y,y')=0$ for $y,y'\in\mathfrak n$. Moreover,
$(\d_\LL\eta)(Z,y)=-\eta(Cy)$. The two-form $-\phi\wedge\Psi_C\eta$
takes the same values, because $\Psi_C\eta$ vanishes on $Z$ and
restricts to $\eta\circ C$ on $\mathfrak n$. The value of a two-form
on $Z$ and $\mathfrak n$ determines it, which proves~\eqref{eq:almost-d}.
\end{proof}

\begin{lemma}\label{lem:wang}
Let $\mathcal J$ be an invariant generalized complex structure on
$\g_A$, and let $\tau$ be as in~\eqref{eq:chi-phi}. Put
\[
\Psi_p=\Psi_C|_{\bigwedge^p\mathfrak n^*}-\tau\id
\qquad(0\le p\le 2n-1),
\]
and let $\Psi_p$ be the zero map of the zero space otherwise.
Then, for $0\le p\le 2n$,
\begin{equation}\label{eq:wang}
GH_{\db}^{n-p}(\g_A)\cong H^p(\LL,\C_\chi)
\cong\ker\Psi_p\oplus\coker\Psi_{p-1}.
\end{equation}
Let $c_1,\ldots,c_{2n-1}$ be the eigenvalues of $C$, with algebraic
multiplicity. The eigenvalues of $\Psi_p$ are then
$\sum_{i\in I}c_i-\tau$, where $I$ runs over index subsets of size $p$.
If $C$ is diagonalizable, so is $\Psi_p$, and $\dim\ker\Psi_p$ is the
number of these sums that vanish.
\end{lemma}

\begin{proof}
The first isomorphism is~\eqref{eq:ACK}. By~\eqref{eq:cochain-splitting},
a $p$-cochain can be written uniquely as $\alpha+\phi\wedge\beta$, with
$\alpha\in\bigwedge^p\mathfrak n^*$ and
$\beta\in\bigwedge^{p-1}\mathfrak n^*$. We use~\eqref{eq:twist},
\eqref{eq:chi-phi} and~\eqref{eq:almost-d}, together with $\d_\LL\phi=0$
and $\phi\wedge\phi=0$. They give
$\d_\LL(\phi\wedge\beta)=-\phi\wedge\d_\LL\beta=0$, and hence
\begin{equation}\label{eq:two-term}
\d_{\LL,\chi}(\alpha+\phi\wedge\beta)
=-\phi\wedge\Psi_C\alpha+\tau\phi\wedge\alpha
=-\phi\wedge\Psi_p\alpha.
\end{equation}
The complex is therefore the direct sum, over $p$, of the two-term
complexes
\[
\bigwedge^p\mathfrak n^*
\xrightarrow{\;-\phi\wedge\Psi_p\;}
\phi\wedge\bigwedge^p\mathfrak n^*,
\]
placed in degrees $p$ and $p+1$. The map $\phi\wedge$ is injective on
$\bigwedge^\bullet\mathfrak n^*$. Hence the summand indexed by $p$
contributes $\ker\Psi_p$ in degree $p$, and the summand indexed by $p-1$
contributes $\coker\Psi_{p-1}$ in degree $p$. This proves~\eqref{eq:wang}. The isomorphism sends $(\alpha,[\beta])$ to the class of
$F_p^{-1}\bigl((\alpha+\phi\wedge\beta)\otimes\rho\bigr)$.

For the spectrum, triangularize $C^T$ on $\mathfrak n^*$
and use~\eqref{eq:Psi-expanded}. The operator $\Psi_C$ then acts
triangularly on wedge monomials, with diagonal entries $\sum_{i\in I}c_i$.
Subtracting $\tau$ proves the spectral assertion. When $C$ is
diagonalizable, an eigenbasis of $C^T$ gives a wedge eigenbasis of
$\Psi_p$, which proves the kernel count.
\end{proof}

\vspace{1em}

\section{Diagonal structures and eigenvalue sums}\label{sec:diagonal}

In this section, we assume that $A$ is diagonalizable over $\R$.

\begin{theorem}[{\cite{PaperI}*{Theorem A}}]
\label{thm:diagonal-existence}
The Lie algebra $\g_A$ admits an invariant generalized complex structure
of type $k$ if and only if, with multiplicities, the eigenvalues of
$A$ can be arranged as
\begin{equation}\label{eq:A-spectrum}
\spec(A)=\{\lambda_1,\lambda_1,\ldots,\lambda_h,\lambda_h,
\mu,\zeta_1,-\zeta_1,\ldots,\zeta_s,-\zeta_s\},
\qquad h+s=n-1,
\end{equation}
where $h,s\ge0$ and either $k=h$ or $k=h+1$.
\end{theorem}

We record and
verify the explicit structures it constructs, because their canonical
characters enter the cohomology. Choose an eigenbasis
$e_1,\ldots,e_{2n-1}$ of $A$ on $\h$, ordered as in~\eqref{eq:A-spectrum},
let $e^0,\ldots,e^{2n-1}$ be the basis dual to $e_0,\ldots,e_{2n-1}$, and put
\begin{align}
    d=2h+1,\qquad r_m=2h+2m\ & (1\le m\le s),\qquad  \theta_i=e^{2i-1}+\ii e^{2i} \ (1 \le i \le h),
\end{align}
\[
\Theta=\theta_1\wedge\cdots\wedge\theta_h.
\]
Empty products are taken to be $1$, and empty sums to e $0$. Define
\begin{equation}\label{eq:models}
\begin{aligned}
\omega_s&=\sum_{m=1}^s e^{r_m}\wedge e^{r_m+1},\\
\rho_{\mathrm I}
&=e^{\ii(e^0\wedge e^d+\omega_s)}\wedge\Theta,
&k_{\mathrm I}&=h,\\
\rho_{\mathrm {II}}
&=e^{\ii\omega_s}\wedge\Theta\wedge(e^0+\ii e^d),
&k_{\mathrm {II}}&=h+1.
\end{aligned}
\end{equation}
In the notation of the pure-spinor normal form, both models have
$\beta=0$. Model~I has $\omega=e^0\wedge e^d+\omega_s$ and
$\Omega=\Theta$. Model~II has $\omega=\omega_s$ and
$\Omega=\Theta\wedge(e^0+\ii e^d)$. The identical pairs supply $h$
complex directions, and the opposite pairs supply $s$ symplectic planes.
The remaining two directions, spanned by $e_0$ and $e_d$, form a
symplectic plane in Model~I and a complex line in Model~II.

\begin{lemma}\label{lem:model-integrability}
Both spinors in~\eqref{eq:models} are non-degenerate and integrable.
For the corresponding model,
\begin{equation}\label{eq:Lambda}
\d\rho=-\Lambda e^0\wedge\rho=c(-\Lambda e^0)\rho,\qquad
\Lambda_{\mathrm I}=\sum_{i=1}^h\lambda_i,
\qquad
\Lambda_{\mathrm {II}}=\mu+\sum_{i=1}^h\lambda_i.
\end{equation}
\end{lemma}

\begin{proof}
Write $Ae_j=a_je_j$. By~\eqref{eq:dgA}, $\d e^0=0$ and
$\d e^j=-a_je^0\wedge e^j$ for $j\ge1$. Hence
\[
\d\omega_s=0,\quad \d(e^0\wedge e^d)=0,\quad
\d\theta_i=-\lambda_i e^0\wedge\theta_i,
\quad
\d(e^0+\ii e^d)=-\mu e^0\wedge(e^0+\ii e^d).
\]
For instance,
$\d(e^{r}\wedge e^{r+1})=-\zeta e^0\wedge e^r\wedge e^{r+1}
+\zeta e^0\wedge e^r\wedge e^{r+1}=0$ when $a_r=-a_{r+1}=\zeta$.
Moving $e^0$ to the front of $\Theta$ produces signs that cancel those
of the Leibniz rule, so $\d\Theta=-\bigl(\sum_i\lambda_i\bigr)e^0\wedge\Theta$.
The graded Leibniz rule then proves~\eqref{eq:Lambda}.

For non-degeneracy, the normal-form condition reads
$\omega^{n-k}\wedge\Omega\wedge\overline\Omega\ne0$. In Model~I,
$n-k=s+1$ and $\omega^{s+1}=(s+1)\,e^0\wedge e^d\wedge\omega_s^s$.
In Model~II, $n-k=s$, and
$(e^0+\ii e^d)\wedge(e^0-\ii e^d)=-2\ii e^0\wedge e^d$. In both cases,
the product is therefore a nonzero multiple of
\[
e^0\wedge e^d\wedge\omega_s^s\wedge\Theta\wedge\overline\Theta,
\]
which is a volume form. Since $\d\rho=c(-\Lambda e^0)\rho$,
Proposition~\ref{prop:integrable-differentials} proves integrability.
\end{proof}

\

For these models, the data $(C,\tau)$ of Lemma~\ref{lem:wang} are
explicit. The following table fixes the normalization of $Z$ and the one
exceptional eigenvalue. The other eigenvalues occur in pairs.

\begin{proposition}\label{prop:specC}
For either spinor in~\eqref{eq:models}, let $\LL=\Ann(\rho)$, and let
$\phi,\mathfrak n,C$ be as in Proposition~\ref{thm:L-almost-abelian},
with $Z$ given by the table below. Then $C$ is diagonalizable and
\begin{equation}\label{eq:C-spectrum}
\Sigma:=\spec(C)
=\{\pm\lambda_1,\ldots,\pm\lambda_h,
\pm\zeta_1,\ldots,\pm\zeta_s,\nu\},
\end{equation}
where
\[
\begin{array}{c|c|c|c|c}
\text{model}&\text{type}& Z&\nu&\Lambda\\\hline
\mathrm I&h&e_0-\ii e^d&\mu&\sum_i\lambda_i\\
\mathrm {II}&h+1&e_0+\ii e_d&-\mu&\mu+\sum_i\lambda_i
\end{array}
\]
The canonical character is $\chi=-\Lambda\phi$; that is, $\tau=-\Lambda$.
If $\g_A$ is unimodular, then $\tr C=-2\Lambda$.
\end{proposition}

\begin{proof}
For a two-form $\gamma$, the Clifford action~\eqref{eq:clifford}
satisfies
\[
c(X+\xi-\iota_X\gamma)(e^\gamma\wedge\alpha)
=e^\gamma\wedge c(X+\xi)\alpha,
\]
because $\iota_X(e^\gamma\wedge\alpha)
=\iota_X\gamma\wedge e^\gamma\wedge\alpha+e^\gamma\wedge\iota_X\alpha$.
If $\Omega$ is decomposable, the two components of
$c(X+\xi)\Omega$ have different degrees. Consequently, for
$\rho=e^{\ii\omega}\wedge\Omega$,
\begin{equation}\label{eq:annihilator}
\Ann(\rho)=\{X-\ii\iota_X\omega+\xi:
\iota_X\Omega=0,\ \xi\wedge\Omega=0\}.
\end{equation}
The vectors $Z$ in the table satisfy these conditions and have
$\phi(Z)=1$. In both models, define
\[
v_i=e_{2i}-\ii e_{2i-1},\qquad
p_m=e_{r_m}-\ii e^{r_m+1},\qquad
q_m=e_{r_m+1}+\ii e^{r_m}.
\]
A basis of $\mathfrak n$ consists of $\theta_i,v_i$ ($1\le i\le h$) and
$p_m,q_m$ ($1\le m\le s$), together with the additional element
\[
w_{\mathrm I}=e_d+\ii e^0,\qquad
w_{\mathrm {II}}=e^0+\ii e^d.
\]
Indeed, each of these elements satisfies~\eqref{eq:annihilator} and has
$\phi=0$. Their vector parts span the solutions $X$ of the vector
conditions with $e^0(X)=0$; this space has dimension $h+2s+1$ in
Model~I and $h+2s$ in Model~II. The remaining pure covectors span the
solutions $\xi$ of $\xi\wedge\Omega=0$, of dimension $h$ and $h+1$,
respectively. Altogether there are $2h+2s+1=2n-1$ independent elements.

By~\eqref{eq:cotangent-bracket}, the coadjoint terms are
\[
\ad_{e_0}^*|_{\h^*}=-A^T,\qquad
\ad_{e_0}^*e^0=0,\qquad
\ad_Y^*\eta=\eta(AY)e^0\quad(Y\in\h_\C).
\]
In both models these formulas give
\[
\br Z{\theta_i}=-\lambda_i\theta_i,\quad
\br Z{v_i}=\lambda_i v_i,\quad
\br Z{p_m}=\zeta_m p_m,\quad
\br Z{q_m}=-\zeta_m q_m.
\]
The possible extra terms $\eta(AY)e^0$ vanish here, since
$e^d(AY)=0$ for $Y=v_i,e_{r_m},e_{r_m+1}$, and $\eta(Ae_d)=\mu\,\eta(e_d)=0$ for $\eta=\theta_i,e^{r_m},e^{r_m+1}$. For the exceptional
elements,
\[
\br{e_0-\ii e^d}{e_d+\ii e^0}
=\mu e_d+\ii\mu e^0=\mu w_{\mathrm I},
\]
and
\[
\br{e_0+\ii e_d}{e^0+\ii e^d}
=-\ii\mu e^d-\mu e^0=-\mu w_{\mathrm {II}}.
\]
Thus the displayed basis is an eigenbasis for $C$.

By~\eqref{eq:Lambda} we may take $z=-\Lambda e^0$ in~\eqref{eq:chi}.
This gives $\chi(X+\xi)=2\langle-\Lambda e^0,X+\xi\rangle
=-\Lambda e^0(X)$, so $\chi=-\Lambda\phi$, in accordance
with~\eqref{eq:chi-phi}. Finally, unimodularity means
$0=\tr A=2\sum_i\lambda_i+\mu$, and $\tr C=\nu$. This gives
$\mu=-2\Lambda_{\mathrm I}$ and $-\mu=-2\Lambda_{\mathrm {II}}$, proving
the trace assertion.
\end{proof}

\

We can now state the dimension formula using only the pair
$(\Sigma,\Lambda)$. It applies to either model, including empty families
and repeated or zero eigenvalues.

\begin{theorem}\label{thm:master}
For either structure~\eqref{eq:models}, list the indexed spectrum of $C$
as $c_1,\ldots,c_{2n-1}$, and set
\[
N_p=\#\left\{I\subseteq\{1,\ldots,2n-1\}:|I|=p,
\ \sum_{i\in I}c_i=-\Lambda\right\}\quad(0\le p\le 2n-1),
\]
with $N_p=0$ otherwise. Then
\begin{equation}\label{eq:master}
\dim GH_{\db}^{n-p}(\g_A)=N_p+N_{p-1}\quad(0\le p\le2n).
\end{equation}
Equivalently, let
$P_\rho(t)=\sum_{p=0}^{2n}\dim GH_{\db}^{n-p}(\g_A)\,t^p$. Then
\begin{equation}\label{eq:genfun}
P_\rho(t)=(1+t)[q^{-\Lambda}]\prod_{i=1}^{2n-1}(1+tq^{c_i}).
\end{equation}
Here the product lies in the group algebra $\Z[\mathcal A][t]$, where
$\mathcal A\subset\R$ is the additive group generated by
$c_1,\ldots,c_{2n-1},\Lambda$. The symbols satisfy
$q^\gamma q^{\gamma'}=q^{\gamma+\gamma'}$, and the bracket denotes
coefficient extraction. If $\tr A=0$, then
\begin{equation}\label{eq:symmetry}
N_p=N_{2n-1-p},\qquad
\dim GH_{\db}^j(\g_A)=\dim GH_{\db}^{-j}(\g_A).
\end{equation}
\end{theorem}

\begin{proof}
By Proposition~\ref{prop:specC}, Lemma~\ref{lem:wang} applies with
$\tau=-\Lambda$ and diagonalizable $C$. For an eigenbasis $\eta_i$ of
$C^T$, one has
\[
\Psi_p(\eta_{i_1}\wedge\cdots\wedge\eta_{i_p})
=\left(\sum_{a=1}^p c_{i_a}+\Lambda\right)
\eta_{i_1}\wedge\cdots\wedge\eta_{i_p}.
\]
Thus $\dim\ker\Psi_p=N_p$. Since $\Psi_{p-1}$ is an endomorphism,
rank--nullity gives $\dim\coker\Psi_{p-1}=N_{p-1}$, and~\eqref{eq:wang}
proves~\eqref{eq:master}. Expanding the product in~\eqref{eq:genfun}
shows that its $q^{-\Lambda}$ coefficient is $\sum_pN_pt^p$.
Multiplication by $1+t$ gives $P_\rho$.

When $\tr A=0$, Proposition~\ref{prop:specC} gives
$\sum_i c_i=-2\Lambda$. A subset has sum $-\Lambda$ exactly when its
complement does, which proves $N_p=N_{2n-1-p}$. Substituting $2n-p$ for $p$
in~\eqref{eq:master} then gives
$\dim GH_{\db}^{p-n}=N_{p-1}+N_p=\dim GH_{\db}^{n-p}$.
\end{proof}

\begin{corollary}\label{cor:generic}
Assume $\tr A=0$ and that the indexed family
$\lambda_1,\ldots,\lambda_h,\zeta_1,\ldots,\zeta_s$ is linearly
independent over $\Q$. For either model,
\[
N_{h+2r}=N_{h+2r+1}=\binom{s}{r}\quad(0\le r\le s),
\qquad N_p=0\text{ otherwise}.
\]
Consequently,
\begin{equation}\label{eq:generic-polynomial}
P_\rho(t)=t^h(1+t)^2(1+t^2)^s,
\end{equation}
and
\[
\sum_j\dim GH_{\db}^j(\g_A)=2^{s+2},\qquad
GH_{\db}^j(\g_A)=0\quad\text{if }|j|>n-h.
\]
To express this in terms of the type $k$, substitute $(h,s)=(k,n-k-1)$
for Model~I and $(h,s)=(k-1,n-k)$ for Model~II.
\end{corollary}

\begin{example}\label{ex:generic}
Suppose $\lambda,\kappa$ are linearly independent over $\Q$ and
$\spec(A)=\{\lambda,-\lambda,\kappa,\kappa,-2\kappa\}$.
For Model~II with $h=s=1$, $\lambda_1=\kappa$, $\zeta_1=\lambda$ and
$\mu=-2\kappa$, the type is $2$. Moreover, $\Lambda=-\kappa$ and
$\Sigma=\{\lambda,-\lambda,\kappa,-\kappa,2\kappa\}$.
By Corollary~\ref{cor:generic},
\[
P_\rho(t)=t(1+t)^2(1+t^2),\qquad
(\dim GH_{\db}^{3},\ldots,\dim GH_{\db}^{-3})=(0,1,2,2,2,1,0).
\]
\end{example}

\vspace{1em}

\section{Comparison with compact quotients}\label{sec:quotients}

Sections~\ref{sec:algebra} and~\ref{sec:diagonal} compute the cohomology
of invariant forms. We now ask when these computations give the
cohomology of a compact quotient. Let
\[
G_A=\R\ltimes_{e^{tA}}\h
\]
be the simply connected Lie group with Lie algebra $\g_A$, let
$\Gamma<G_A$ be a lattice, and put $X=\Gamma\backslash G_A$. A Lie group
with a lattice is unimodular. Hence $\tr A=0$ throughout this section,
and the symmetry~\eqref{eq:symmetry} holds whenever
Theorem~\ref{thm:master} applies.

An invariant generalized complex structure $\mathcal J$ on $\g_A$
descends to $X$. We keep the notation $L_X$, $K_X$, $U^j(X)$ and
$H^p(L_X,K_X)$ of Section~\ref{sec:spinors}. The grading of forms on $X$
is the left translate of the invariant grading, so an invariant basis of
$\UU^j$ is a global frame:
\begin{equation}\label{eq:invariant-frame}
U^j(X)=C^\infty(X,\C)\otimes_\C\UU^j .
\end{equation}
The inclusion $\UU^\bullet\subset U^\bullet(X)$ commutes with $\partial$
and $\db$. Through the maps $F_p$, it is the inclusion
\[
C^\bullet(\LL,K)\hookrightarrow C^\bullet(L_X,K_X)
\]
of invariant cochains, and it induces maps
$GH^j_{\db}(\g_A)\to GH^j_{\db}(X)$ compatible with~\eqref{eq:ACK}
and~\eqref{eq:global-ACK}.

The comparison proceeds in four steps. Averaging over $X$ shows that
invariant forms are always a direct summand, and reduces the four
cohomologies of Section~\ref{sec:spinors} to generalized Dolbeault
cohomology (Section~\ref{subsec:averaging}). For the structures given
by the quotient construction of~\cite{PaperI}, which include Model~II,
the eigenspace $\LL$ is an extension of the eigenspace of a complex
structure on a quotient algebra $\mathfrak b$ by an abelian ideal
(Section~\ref{subsec:adapted}). On $X$ this extension is realized by a
symplectic torus bundle over a compact complex manifold $B$, and two
spectral sequences reduce the comparison to Dolbeault cohomology of $B$
with coefficients in the cohomology of the fibers
(Sections~\ref{subsec:torus-bundle}--\ref{subsec:comparison}). Finally,
these coefficients split into characters indexed by eigenvalue sums of
$A$, as in Lemma~\ref{lem:wang}. Whether they are computed by invariant
forms depends on $\Gamma$, and not only on $\g_A$ and $\mathcal J$
(Section~\ref{subsec:characters}).

\subsection{Averaging and reduction to \texorpdfstring{$\db$}{dbar}}\label{subsec:averaging}

Fix an invariant volume form $\mathrm{vol}_X$ on $X$ with
$\int_X\mathrm{vol}_X=1$, and an invariant coframe
$\xi^1,\ldots,\xi^{2n}$ of $\g_A$. Every complex form on $X$ can be
written uniquely as $\sum_If_I\xi^I$ with $f_I\in C^\infty(X,\C)$.
Define
\[
\operatorname{Av}\Bigl(\sum_If_I\xi^I\Bigr)
=\sum_I\Bigl(\int_Xf_I\,\mathrm{vol}_X\Bigr)\xi^I .
\]

\begin{lemma}\label{lem:averaging}
The operator $\operatorname{Av}$ commutes with $\d$, maps $U^j(X)$ onto
$\UU^j$, and commutes with $\partial$ and $\db$. Consequently, putting
$Q^j=\ker(\operatorname{Av}|_{U^j(X)})$,
\begin{equation}\label{eq:averaging-splitting}
U^\bullet(X)=\UU^\bullet\oplus Q^\bullet
\end{equation}
is a decomposition into subcomplexes for both $\partial$ and $\db$, and
$\overline{Q^j}=Q^{-j}$.
\end{lemma}

\begin{proof}
Let $X_1,\ldots,X_{2n}$ be the invariant frame dual to
$\xi^1,\ldots,\xi^{2n}$. Then
\[
\d(f\xi^I)=\sum_aX_a(f)\,\xi^a\wedge\xi^I+f\,\d\xi^I,
\]
and $\d\xi^I$ is invariant. By unimodularity,
$\mathcal L_{X_a}\mathrm{vol}_X=-\tr(\ad_{X_a})\mathrm{vol}_X=0$, so
Stokes' theorem gives
\[
\int_XX_a(f)\,\mathrm{vol}_X
=\int_X\d\iota_{X_a}(f\,\mathrm{vol}_X)=0 .
\]
Hence
$\operatorname{Av}\d(f\xi^I)=\bigl(\int_Xf\,\mathrm{vol}_X\bigr)\d\xi^I
=\d\operatorname{Av}(f\xi^I)$.
By~\eqref{eq:invariant-frame},
$\operatorname{Av}(f\alpha)=\bigl(\int_Xf\,\mathrm{vol}_X\bigr)\alpha$
for $\alpha\in\UU^j$. Thus $\operatorname{Av}$ maps $U^j(X)$ onto
$\UU^j$ and commutes with the projections $\operatorname{pr}_j$. By
\eqref{eq:projected-differentials}, it commutes with $\partial$ and
$\db$. Since $\operatorname{Av}$ is a projection onto $\UU^\bullet$, the
decomposition follows. The last assertion follows from the reality of
$\operatorname{Av}$ and $\overline{U^j(X)}=U^{-j}(X)$.
\end{proof}

\begin{proposition}[{\cite{ACK}*{Corollary 5.2}}]\label{prop:reduction-to-dbar}
For each of $GH_{\db}$, $GH_\partial$, $GH_{BC}$ and $GH_A$, the
inclusion of invariant forms induces an injection
$GH^j(\g_A)\hookrightarrow GH^j(X)$. Its cokernel is the corresponding
cohomology of $Q^\bullet$ in degree $j$. If $H_{\db}(Q^\bullet)=0$, all
four injections are isomorphisms.
\end{proposition}

\subsection{Structures fibered over a complex quotient}\label{subsec:adapted}

To go beyond injectivity, we use structures that split into symplectic
directions along an ideal and complex directions on the quotient. At the
Lie algebra level, they are given by the following construction
of~\cite{PaperI}.

\begin{proposition}[\cite{PaperI}*{Theorem C}]
\label{prop:bundle-construction}
Suppose $\h=\h_c\oplus\mathfrak k$ is a sum of $A$-invariant subspaces
with $\dim\h_c=2k-1$, $\dim\mathfrak k=2\ell$ and $k+\ell=n$ (so
$k\ge1$). Then $\mathfrak k$ is an abelian ideal of $\g_A$, and
$\mathfrak b=\g_A/\mathfrak k$ is the almost abelian Lie algebra
$\g_{A_c}$ with $A_c=A|_{\h_c}$. Assume that $\mathfrak b$ carries an
integrable complex structure $J_B$, with
$\mathfrak b_\C=\mathfrak b^{1,0}\oplus\mathfrak b^{0,1}$, and that a
closed real two-form $\omega$ on $\g_A$ restricts to a symplectic form
$\sigma$ on $\mathfrak k$. Let $\pi:\g_A\to\mathfrak b$ be the quotient
map, let $0\ne\Omega\in\bigwedge^k\mathfrak b^{*1,0}$, and write
$\d_{\mathfrak b}\Omega=\alpha_B\wedge\Omega$ with
$\alpha_B\in\mathfrak b^{*0,1}$. Then
\begin{equation}\label{eq:bundle-spinor}
\rho=e^{\ii\omega}\wedge\pi^*\Omega
\end{equation}
defines an invariant generalized complex structure of type $k$, and
\begin{equation}\label{eq:bundle-character}
\d\rho=c(\pi^*\alpha_B)\rho,\qquad
\chi(X+\xi)=\alpha_B(\pi X)\quad(X+\xi\in\LL).
\end{equation}
Moreover, $A|_{\mathfrak k}\in\mathfrak{sp}(\mathfrak k,\sigma)$.
\end{proposition}

\begin{lemma}\label{lem:L-extension}
In the setting of Proposition~\ref{prop:bundle-construction}, regard
$J_B$ as a generalized complex structure on $\mathfrak b$. Its eigenspace
and canonical line are
\[
\LL_B=\mathfrak b^{0,1}\oplus\mathfrak b^{*1,0},\qquad
K_{\mathfrak b}=\bigwedge^k\mathfrak b^{*1,0}=\C\Omega,
\]
and by~\eqref{eq:chi} its canonical character is
$\chi_B(Y+\vartheta)=\alpha_B(Y)$. Put
\[
\mathcal B(X+\xi)=X+\xi-\ii\iota_X\omega,\qquad
\mathfrak S=\mathcal B(\mathfrak k_\C)
=\{u-\ii\iota_u\omega:u\in\mathfrak k_\C\}.
\]
Then the following hold.
\begin{enumerate}
\item $\mathcal B$ is an automorphism of the Lie algebra
$\left(E=(\g\oplus\g^*)_\C,\br\cdot\cdot\right)$, and $\LL=\mathcal B(\LL_0)$, where
\[
\LL_0=\{X+\pi^*\vartheta:\ \pi X\in\mathfrak b^{0,1},\
\vartheta\in\mathfrak b^{*1,0}\}.
\]
\item The map $\varpi:\LL\to\LL_B$ given by
$\varpi\bigl(\mathcal B(X+\pi^*\vartheta)\bigr)=\pi X+\vartheta$ is a
surjective Lie algebra homomorphism with kernel $\mathfrak S$. The ideal
$\mathfrak S$ is abelian and is contained in the ideal $\mathfrak n$ of
Proposition~\ref{thm:L-almost-abelian}.
\item $\chi=\chi_B\circ\varpi$, and $\Omega\mapsto
e^{\ii\omega}\wedge\pi^*\Omega$ is an isomorphism of $\LL$-modules
$\varpi^*K_{\mathfrak b}\cong K$.
\item For $u\in\mathfrak k_\C$, $\vartheta\in\mathfrak b^{*1,0}$, and
$\widetilde Y\in\g_{A,\C}$ with $\pi\widetilde Y\in\mathfrak b^{0,1}$,
\[
\br{\mathcal B\widetilde Y}{\mathcal Bu}
=e^0(\widetilde Y)\,\mathcal B(Au),\qquad
\br{\mathcal B\pi^*\vartheta}{\mathcal Bu}=0 .
\]
\end{enumerate}
\end{lemma}

\begin{proof}
The identity used in the proof of Proposition~\ref{prop:specC}, applied
to the closed two-form $\ii\omega$, gives
$c(\mathcal Bv)(e^{\ii\omega}\wedge\beta)=e^{\ii\omega}\wedge c(v)\beta$
for every form $\beta$. Since $\d\omega=0$, the operator
$e^{\ii\omega}\wedge$ commutes with $\d$. By~\eqref{eq:Cartan-operator},
conjugation by $e^{\ii\omega}\wedge$ sends $D_u$ to $D_{\mathcal Bu}$,
and~\eqref{eq:D-identities} gives
$c(\br{\mathcal Bu}{\mathcal Bv})=c(\mathcal B\br uv)$. As $c$ is
injective on $E$, $\mathcal B$ preserves the bracket. The same identity
gives $\LL=\Ann(\rho)=\mathcal B\bigl(\Ann(\pi^*\Omega)\bigr)$. The two
components of $c(X+\xi)\pi^*\Omega$ have different degrees. Hence $X+\xi$
annihilates $\pi^*\Omega$ exactly when $\iota_{\pi X}\Omega=0$ and
$\xi\wedge\pi^*\Omega=0$, that is, when $\pi X\in\mathfrak b^{0,1}$ and
$\xi\in\pi^*\mathfrak b^{*1,0}$. This proves the first part.

Since $\pi$ is a homomorphism,
$(\ad_X^*\pi^*\vartheta)(W)=-\vartheta([\pi X,\pi W])$, that is,
$\ad_X^*\pi^*\vartheta=\pi^*\ad_{\pi X}^*\vartheta$.
By~\eqref{eq:cotangent-bracket}, the map
$X+\pi^*\vartheta\mapsto\pi X+\vartheta$ is therefore a homomorphism
$\LL_0\to\LL_B$. It is surjective, and its kernel is $\mathfrak k_\C$
because $\pi^*$ is injective. Composing with $\mathcal B^{-1}$ gives
$\varpi$. The ideal $\mathfrak S$ is abelian because $\mathfrak k_\C$ is.
Since $\mathcal B$ does not change vector parts,
$\phi(\mathcal Bu)=e^0(u)=0$ for $u\in\mathfrak k_\C$, so
$\mathfrak S\subseteq\mathfrak n$.

For $v=X+\pi^*\vartheta\in\LL_0$ we have
$c(v)\pi^*=\pi^*c(\pi X+\vartheta)$ and $\d\pi^*=\pi^*\d_{\mathfrak b}$,
hence $D_v\pi^*=\pi^*D_{\varpi\mathcal Bv}$. Together with the first
part,
\[
D_{\mathcal Bv}(e^{\ii\omega}\wedge\pi^*\Omega)
=e^{\ii\omega}\wedge\pi^*\bigl(D_{\varpi\mathcal Bv}\Omega\bigr)
=\chi_B(\varpi\mathcal Bv)\,e^{\ii\omega}\wedge\pi^*\Omega .
\]
By~\eqref{eq:character}, this proves the third part.

Finally, $u\in\h_\C$ and $[\h,\h]=0$ give
$\br{\widetilde Y}{u}=[\widetilde Y,u]=e^0(\widetilde Y)Au$.
By~\eqref{eq:cotangent-bracket} and the formula above,
$\br{\pi^*\vartheta}{u}=-\ad_u^*\pi^*\vartheta
=-\pi^*\ad^*_{\pi u}\vartheta=0$. Applying the automorphism $\mathcal B$
proves the last part.
\end{proof}

\subsection{The torus bundle and its coefficient systems}
\label{subsec:torus-bundle}

Keep the setting of Proposition~\ref{prop:bundle-construction}. The
subgroup $G_{\mathfrak k}=\{0\}\times\mathfrak k$ of $G_A$ is closed
normal and abelian with Lie algebra $\mathfrak k$. The quotient
$G_{\mathfrak b}=G_A/G_{\mathfrak k}=\R\ltimes_{e^{tA_c}}\h_c$ is the
simply connected group with Lie algebra $\mathfrak b$. Let
$\bar\pi:G_A\to G_{\mathfrak b}$ be the quotient map, and assume
\begin{equation}\label{eq:lattice-hypotheses}
\Gamma_{\mathfrak k}=\Gamma\cap G_{\mathfrak k}\text{ is a lattice in }
G_{\mathfrak k},
\qquad
\Gamma_{\mathfrak b}=\bar\pi(\Gamma)\text{ is a lattice in }
G_{\mathfrak b}.
\end{equation}
Then
\begin{equation}\label{eq:torus-bundle}
\mathbb T=\Gamma_{\mathfrak k}\backslash G_{\mathfrak k}\longrightarrow
X\xrightarrow{\ p\ }B=\Gamma_{\mathfrak b}\backslash G_{\mathfrak b}
\end{equation}
is a fiber bundle, with fibers $2\ell$-dimensional tori, over a compact
complex manifold $B$ of complex dimension $k$, whose complex structure is
induced by $J_B$. The form $\omega$ descends to a closed form on $X$ that
restricts to an invariant symplectic form on every fiber. Thus
\eqref{eq:torus-bundle} is a symplectic torus bundle over a complex base.
Invariant forms on $\mathfrak b$ descend to $B$, and we denote them by the same
symbols. In particular, $e^0$ vanishes on $\mathfrak k$ and defines a
closed one-form on $B$.

The invariant form $\Omega$ is a nowhere vanishing section of the
canonical bundle $K_B$, and $\rho=e^{\ii\omega}\wedge p^*\Omega$
generates $K_X$. Hence
\begin{equation}\label{eq:canonical-pullback}
p^*K_B\xrightarrow{\ \sim\ }K_X,\qquad
p^*s\longmapsto e^{\ii\omega}\wedge p^*s .
\end{equation}
Let $L_B=T^{0,1}B\oplus T^{*1,0}B$ be the Lie algebroid of the complex
structure of $B$. Its canonical module $K_B$ has connection
$\nabla^{K_B}_us=c(u)\d s$; this is the holomorphic structure along
$T^{0,1}B$, and it vanishes along $T^{*1,0}B$. The pointwise identities
of Lemma~\ref{lem:L-extension} show that~\eqref{eq:canonical-pullback}
is an isomorphism of modules. For vertical $Y$,
$\nabla^K_{\mathcal BY}(e^{\ii\omega}\wedge p^*s)=0$. If $\widetilde u$
is a projectable lift of a section $u$ of $L_B$, meaning that its vector
part is $p$-related to that of $u$ and its covector part is pulled back
from $B$, then
\[
\nabla^K_{\mathcal B\widetilde u}(e^{\ii\omega}\wedge p^*s)
=e^{\ii\omega}\wedge p^*\bigl(\nabla^{K_B}_us\bigr).
\]

\medskip

Let $M$ be a finite-dimensional $\mathfrak b^{0,1}$-module. Choose a
basis $\bar Y_1,\ldots,\bar Y_k$ of $\mathfrak b^{0,1}$ with dual basis
$\bar\xi^1,\ldots,\bar\xi^k$ of $\mathfrak b^{*0,1}$. Let $\mathcal M$
be the trivial bundle $B\times M$ with
\[
\db_{\mathcal M}(f\otimes m)=\db_Bf\otimes m
+f\sum_{a=1}^k\bar\xi^a\otimes\bar Y_a\cdot m,
\]
where $\db_B$ is the Dolbeault operator of $B$. Since $M$ is a module
and $\mathfrak b^{0,1}$ is a subalgebra, $\db_{\mathcal M}^2=0$, so
$\mathcal M$ is a holomorphic vector bundle. Its smooth $(a,b)$-forms
are $C^\infty(B,\C)\otimes\bigwedge^a\mathfrak b^{*1,0}\otimes
\bigwedge^b\mathfrak b^{*0,1}\otimes M$. The invariant ones form the
subcomplex
\[
\bigwedge^a\mathfrak b^{*1,0}\otimes\bigwedge^\bullet\mathfrak b^{*0,1}
\otimes M
=C^\bullet\bigl(\mathfrak b^{0,1},
\textstyle\bigwedge^a\mathfrak b^{*1,0}\otimes M\bigr),
\]
on which $\db_{\mathcal M}$ is the differential~\eqref{eq:CE}. Here
$\mathfrak b^{0,1}$ acts on $\mathfrak b^{*1,0}$ by the coadjoint
action, which preserves $\mathfrak b^{*1,0}$ because $\mathfrak b^{0,1}$
is a subalgebra. We write
\begin{equation}\label{eq:coefficient-Dolbeault}
H^{a,b}_{\db}(\mathfrak b;M)
=H^b\bigl(\mathfrak b^{0,1},
\textstyle\bigwedge^a\mathfrak b^{*1,0}\otimes M\bigr).
\end{equation}
The inclusion of invariant forms and the Dolbeault isomorphism give the
comparison map
\begin{equation}\label{eq:comparison-map}
H^{a,b}_{\db}(\mathfrak b;M)\longrightarrow
H^b(B,\Omega^a_B\otimes\mathcal M).
\end{equation}
For $M=\C$, this is the usual comparison between invariant and ordinary
Dolbeault cohomology of $B$.

\medskip

Let $\mathcal H^q \to B$ be the flat bundle with fiber
$H^q(p^{-1}(y);\C)$ at $y\in B$ and Gauss--Manin connection $\nabla$. Flatness
makes $\nabla^{0,1}$ a holomorphic structure. The invariant model is
\[
V^q=\bigwedge^q\mathfrak k_\C^*\cong H^q(\mathfrak k_\C;\C).
\]
For $Y\in\mathfrak b_\C$ with a lift $\widetilde Y\in\g_{A,\C}$, put
\begin{equation}\label{eq:fiber-action}
(Y\cdot\eta)(u_1,\ldots,u_q)
=-\sum_{j=1}^q\eta(u_1,\ldots,[\widetilde Y,u_j],\ldots,u_q)
=-e^0(Y)\,(\Psi_{\mathfrak k}\eta)(u_1,\ldots,u_q).
\end{equation}
Here $\Psi_{\mathfrak k}$ is the derivation of
$\bigwedge^\bullet\mathfrak k_\C^*$ extending
$\eta\mapsto\eta\circ A|_{\mathfrak k}$, as
in~\eqref{eq:Psi-expanded}. The second equality holds because
$[\widetilde Y,u]=e^0(\widetilde Y)Au$ for $u\in\mathfrak k_\C$. In
particular, the action does not depend on the lift. Restricted to
$\mathfrak b^{0,1}$, it makes $V^q$ a $\mathfrak b^{0,1}$-module. By the
last part of Lemma~\ref{lem:L-extension}, this is the action of $\LL_B$
on $\bigwedge^q\mathfrak S^*\cong V^q$ induced by the bracket of $\LL$,
with $\mathfrak b^{*1,0}$ acting trivially. We write $\mathcal V^q$ for
the corresponding holomorphic bundle on $B$.

\begin{lemma}\label{lem:GM}
Extend $\eta\in V^q$ to $\g_{A,\C}$ by zero on $\R e_0\oplus\h_c$. The
resulting invariant form restricts to a closed form on each fiber of
$p$, and its fiberwise class $s_\eta$ is a smooth section of
$\mathcal H^q$. As $\eta$ runs over a basis of $V^q$, these sections
form a frame. For $Y\in\mathfrak b_\C$ and its invariant vector field
$Y^L$ on $B$,
\begin{equation}\label{eq:GM-frame}
\nabla_{Y^L}(fs_\eta)=Y^L(f)\,s_\eta+f\,s_{Y\cdot\eta}.
\end{equation}
Consequently, $\eta\mapsto s_\eta$ induces an isomorphism of holomorphic
bundles $\mathcal V^q\cong(\mathcal H^q,\nabla^{0,1})$.
\end{lemma}

\begin{proof}
The fiber through $\Gamma g$ is the image of $gG_{\mathfrak k}$, and
the extended form restricts there to an invariant form on a torus. Such
forms are closed and represent each cohomology class exactly once. This
proves the first two assertions. The Gauss--Manin derivative
$\nabla_Y[\tilde s]$ is represented by the fiberwise restriction of
$\mathcal L_{\widetilde Y}\tilde s$, for any lift $\widetilde Y$ of $Y$.
Take the invariant vector field of a lift $\widetilde Y\in\g_{A,\C}$.
On invariant arguments $u_j\in\mathfrak k_\C$, the invariant form
$\mathcal L_{\widetilde Y}\eta$ takes the value given
by~\eqref{eq:fiber-action}. The Leibniz rule gives~\eqref{eq:GM-frame}.
Taking $Y\in\mathfrak b^{0,1}$ and comparing with $\db_{\mathcal V^q}$
proves the last assertion.
\end{proof}

\subsection{The two spectral sequences}\label{subsec:spectral}

\begin{proposition}\label{prop:spectral-sequences}
Assume that $\mathcal J$ is given by~\eqref{eq:bundle-spinor} and that
$\Gamma$ satisfies~\eqref{eq:lattice-hypotheses}. Consider the filtration
\begin{equation}\label{eq:filtration}
F^pC^m=\{f\in C^m:\ f(v_1,\ldots,v_m)=0\text{ whenever at least }
m-p+1\text{ arguments lie in }\mathfrak S\}
\end{equation}
of $C^\bullet(\LL,K)$, and its analogue on $C^\bullet(L_X,K_X)$ with
$\mathfrak S$ replaced by $S_X=\mathcal B\bigl(T(X/B)_\C\bigr)$. They
give first-quadrant spectral sequences
\begin{align}
E_2^{p,q}&=H^p\bigl(L_B,K_B\otimes\mathcal H^q\bigr)
\ \Longrightarrow\ H^{p+q}(L_X,K_X)\cong GH_{\db}^{n-p-q}(X),
\label{eq:global-SS}\\
\widetilde E_2^{p,q}&=H^p\bigl(\LL_B,K_{\mathfrak b}\otimes V^q\bigr)
\ \Longrightarrow\ H^{p+q}(\LL,K)\cong GH_{\db}^{n-p-q}(\g_A).
\label{eq:invariant-SS}
\end{align}
Here $L_B$ acts on $K_B\otimes\mathcal H^q$ through the holomorphic
structures along $T^{0,1}B$, and $T^{*1,0}B$ acts trivially. The action
of $\LL_B$ on $K_{\mathfrak b}\otimes V^q$ is its invariant restriction.
The inclusion of invariant cochains is filtered and induces a morphism
$\widetilde E_r^{p,q}\to E_r^{p,q}$. Under the identifications
\begin{equation}\label{eq:second-pages}
E_2^{p,q}\cong\bigoplus_{\substack{0\le a,b\le k\\ a-b=k-p}}
H^b\bigl(B,\Omega^a_B\otimes\mathcal H^q\bigr),\qquad
\widetilde E_2^{p,q}\cong\bigoplus_{\substack{0\le a,b\le k\\ a-b=k-p}}
H^{a,b}_{\db}\bigl(\mathfrak b;V^q\bigr),
\end{equation}
the morphism at the second page is the direct sum of the comparison
maps~\eqref{eq:comparison-map} for $M=V^q$, composed with the
isomorphism of Lemma~\ref{lem:GM}.
\end{proposition}

\begin{proof}
By Lemma~\ref{lem:L-extension},
$\mathfrak S$ is an abelian ideal of $\LL$ and
$\chi|_{\mathfrak S}=\chi_B\circ\varpi|_{\mathfrak S}=0$. The
filtration~\eqref{eq:filtration} is the Hochschild--Serre filtration of
the ideal $\mathfrak S$. The resulting spectral sequence has
\[
\widetilde E_1^{p,q}
=H^q\bigl(\mathfrak S,\textstyle\bigwedge^p\LL_B^*\otimes K\bigr),
\qquad
\widetilde E_2^{p,q}=H^p\bigl(\LL_B,H^q(\mathfrak S,K)\bigr).
\]
Since $\mathfrak S$ is abelian and acts trivially on $K$ and on
$\LL/\mathfrak S$, the differentials of its cochain complexes vanish.
Hence $H^q(\mathfrak S,K)=\bigwedge^q\mathfrak S^*\otimes K$. The map
$\mathcal B$ identifies $\bigwedge^q\mathfrak S^*$ with $V^q$, and
Lemma~\ref{lem:L-extension} identifies the induced $\LL_B$-action
with~\eqref{eq:fiber-action}, where $\mathfrak b^{*1,0}$ acts trivially.
The same lemma identifies $K$ with $\varpi^*K_{\mathfrak b}$. This
gives~\eqref{eq:invariant-SS}, and~\eqref{eq:ACK} identifies the
abutment.

The form $\omega$ is closed on $X$, so
$\mathcal B$ is an automorphism of the complexified Courant bracket on
sections \cite{Gualtieri2011}. Pointwise, the first part of
Lemma~\ref{lem:L-extension} gives $L_X=\mathcal B(L_{0,X})$, where
\[
L_{0,X}=\{Y+\xi:\ p_*Y\in T^{0,1}B,\ \xi\in p^*T^{*1,0}B\}.
\]
The vertical bundle $T(X/B)_\C\subset L_{0,X}$ is a subalgebroid.
Projectable sections of $L_{0, X}$ induce the anchor and bracket of $L_B$.
Hence $S_X$ is a subalgebroid of $L_X$ with $L_X/S_X\cong p^*L_B$, and
the differential preserves~\eqref{eq:filtration}. The graded pieces are
\[
\operatorname{gr}^pC^{p+q}
=\Gamma\bigl(X,\textstyle\bigwedge^qS_X^*\otimes
p^*(\bigwedge^pL_B^*\otimes K_B)\bigr)
\]
by~\eqref{eq:canonical-pullback}. The bracket of a vertical vector field
with a projectable section is vertical. By~\eqref{eq:canonical-pullback},
$S_X$ also annihilates the pulled-back frames of $K_X$. Thus $S_X$ acts
on the coefficients by the flat partial connection whose parallel sections
are pulled back from $B$, and $d_0$ is the vertical de Rham differential
with coefficients in $p^*(\bigwedge^pL_B^*\otimes K_B)$. As the fibers
are compact, fiberwise Hodge theory gives
\[
E_1^{p,q}=\Gamma\bigl(B,\textstyle\bigwedge^pL_B^*\otimes K_B
\otimes\mathcal H^q\bigr).
\]
The differential $d_1$ is induced by projectable lifts. Along a lift of
a section of $T^{0,1}B$, it differentiates fiberwise classes by the
Gauss--Manin connection. The lift $\mathcal B(p^*\vartheta)$ of a section
$\vartheta$ of $T^{*1,0}B$ has zero bracket with vertical sections and
acts trivially. Together with~\eqref{eq:canonical-pullback}, this
identifies $(E_1,d_1)$ with $C^\bullet(L_B,K_B\otimes\mathcal H^q)$ and
gives~\eqref{eq:global-SS}. The abutment is identified
by~\eqref{eq:global-ACK}. This is the spectral sequence
of~\cite{ACK}*{Theorem~2.1}, with coefficient bundle $p^*K_B$ and
ordinary vertical degree $q$.

 The elements of $\mathfrak S$ are invariant sections
spanning $S_X$ at every point. Hence invariant cochains in
$F^pC^m(\LL,K)$ span $F^pC^m(L_X,K_X)$ over $C^\infty(X)$, and the
inclusion of invariant cochains is filtered. On the first page, it is the
inclusion of invariant sections, with $\eta\in V^q$ sent to $s_\eta$. By
Lemma~\ref{lem:GM}, it commutes with $d_1$.

 Let $M$ be a $\mathfrak b^{0,1}$-module,
regarded as an $\LL_B$-module on which $\mathfrak b^{*1,0}$ acts
trivially. By~\eqref{eq:cotangent-bracket}, $\LL_B$ is the semidirect
product of the subalgebra $\mathfrak b^{0,1}$ and the abelian ideal
$\mathfrak b^{*1,0}$, on which $\mathfrak b^{0,1}$ acts by the coadjoint
action. The ideal acts trivially on $K_{\mathfrak b}\otimes M$ because
$\chi_B$ vanishes on it. The differential~\eqref{eq:CE} therefore has no
terms coming from the action or the brackets of the ideal, and
\[
C^p(\LL_B,K_{\mathfrak b}\otimes M)
=\bigoplus_{r+b=p}C^b\bigl(\mathfrak b^{0,1},
\textstyle\bigwedge^r(\mathfrak b^{*1,0})^*\otimes K_{\mathfrak b}
\otimes M\bigr)
\]
as complexes. For $\bar Y\in\mathfrak b^{0,1}$ we have
$D_{\bar Y}\Omega=\iota_{\bar Y}\d\Omega=\mathcal L_{\bar Y}\Omega$.
Thus $K_{\mathfrak b}$ is the top exterior power of the coadjoint module
$\mathfrak b^{*1,0}$, and contraction is a $\mathfrak b^{0,1}$-equivariant
isomorphism
$\bigwedge^r(\mathfrak b^{*1,0})^*\otimes K_{\mathfrak b}\cong
\bigwedge^{k-r}\mathfrak b^{*1,0}$. With $a=k-r$, we obtain
\begin{equation}\label{eq:canonical-to-Dolbeault}
H^p(\LL_B,K_{\mathfrak b}\otimes M)\cong
\bigoplus_{a-b=k-p}H^{a,b}_{\db}(\mathfrak b;M).
\end{equation}
The same argument with smooth sections uses the holomorphic isomorphism
$\bigwedge^rT^{1,0}B\otimes K_B\cong\Omega^{k-r}_B$. It identifies
$H^p(L_B,K_B\otimes\mathcal M)$ with
$\bigoplus_{a-b=k-p}H^b(B,\Omega^a_B\otimes\mathcal M)$, compatibly
with~\eqref{eq:comparison-map}. For $M=V^q$ and
$\mathcal M=\mathcal V^q\cong\mathcal H^q$, this
proves~\eqref{eq:second-pages} and the description of the morphism.
\end{proof}

\subsection{The comparison theorem}\label{subsec:comparison}

The two spectral sequences are related by averaging. Transport
$\operatorname{Av}$ to cochains through the maps $F_p$; since these are
$C^\infty(X)$-linear and preserve invariance, this averages the
coefficients of a cochain in an invariant frame. Invariant cochains in
$F^pC^m(\LL,K)$ span $F^pC^m(L_X,K_X)$ over $C^\infty(X)$, so
$\operatorname{Av}$ preserves the filtration~\eqref{eq:filtration}. By
Lemma~\ref{lem:averaging}, $C^\bullet(L_X,K_X)$ is the direct sum of
the filtered subcomplexes of invariant cochains and of cochains
corresponding to $Q^\bullet$. Therefore
\begin{equation}\label{eq:SS-splitting}
E_r^{p,q}=\widetilde E_r^{p,q}\oplus E_r^{p,q}(Q)\qquad(r\ge0),
\end{equation}
where $E_r(Q)$ is the spectral sequence of the filtered complex
corresponding to $Q^\bullet$, which converges to $H_{\db}(Q)$. The
morphism of Proposition~\ref{prop:spectral-sequences} is the inclusion
of the first summand.

\begin{theorem}\label{thm:invariance}
Assume that $\mathcal J$ is given by~\eqref{eq:bundle-spinor} and that
$\Gamma$ satisfies~\eqref{eq:lattice-hypotheses}. Suppose that the
comparison maps
\begin{equation}\label{eq:comparison-hypothesis}
H_{\db}^{a,b}(\mathfrak b;V^q)
\longrightarrow H^b(B,\Omega_B^a\otimes\mathcal H^q)
\end{equation}
are isomorphisms for all $0\le a,b\le k$ and $0\le q\le2\ell$. Then the
inclusion of invariant forms induces isomorphisms
\[
GH_{\db}^j(\g_A)\xrightarrow{\ \sim\ }GH_{\db}^j(X)
\qquad\text{for all }j.
\]
\end{theorem}

\begin{proof}
By Proposition~\ref{prop:spectral-sequences} and the hypothesis, the
inclusion $\widetilde E_2\to E_2$ is an isomorphism. By
\eqref{eq:SS-splitting}, $E_2(Q)=0$. The filtration is finite, with
$0\le p\le2k$ and $0\le q\le2\ell$, so $H_{\db}(Q)=0$.
Proposition~\ref{prop:reduction-to-dbar} completes the proof.
\end{proof}

\begin{corollary}\label{cor:BC-invariance}
Under the hypotheses of Theorem~\ref{thm:invariance}, invariant forms
also compute $GH_\partial^j(X)$, $GH_{BC}^j(X)$ and $GH_A^j(X)$ for
every $j$.
\end{corollary}

\begin{proof}
The proof of Theorem~\ref{thm:invariance} shows that $H_{\db}(Q)=0$.
Apply Proposition~\ref{prop:reduction-to-dbar}.
\end{proof}

The hypothesis~\eqref{eq:comparison-hypothesis} concerns the base and
the lattice. For $q=0$, it is the untwisted comparison of invariant and
ordinary Dolbeault cohomology of $B$. For $q>0$, the
representations~\eqref{eq:fiber-action} are generally nontrivial and need
not be nilpotent. The nilpotent-module argument
of~\cite{ACK}*{Section~5} therefore does not apply automatically. Conversely,
by~\eqref{eq:SS-splitting}, a failure of~\eqref{eq:comparison-hypothesis}
produces a nonzero page $E_2(Q)$. Whether it survives to $H_{\db}(Q)$
depends on the higher differentials; see Example~\ref{ex:lattice-dependence}.

\subsection{Characters from eigenvalue sums and Model~II}
\label{subsec:characters}

By~\eqref{eq:fiber-action}, the algebra $\mathfrak b^{0,1}$ acts on $V^q$
through the character $e^0|_{\mathfrak b^{0,1}}$ and the single operator
$-\Psi_{\mathfrak k}$. As in Lemma~\ref{lem:wang}, the eigenvalues of
$\Psi_{\mathfrak k}$ on $V^q$ are sums of $q$ eigenvalues of
$A|_{\mathfrak k}$. The coefficient condition therefore reduces to line
bundles, mirroring the approach established by \cite{Kasuya2014} for computing the Dolbeault cohomology of solvmanifolds using local systems.

For $\kappa\in\C$, let $\mathcal E_\kappa\to B$ be the holomorphic line
bundle attached to the one-dimensional module $\C_{-\kappa e^0}$ of
$\mathfrak b^{0,1}$. It is the trivial line bundle with
\begin{equation}\label{eq:E-kappa}
\db_\kappa f=\db_Bf-\kappa f\,(e^0)^{0,1},
\end{equation}
where $(e^0)^{0,1}$ is the $(0,1)$-component of the closed real one-form
$e^0$ on $B$.

\begin{lemma}\label{lem:characters}
Let $\kappa_1,\ldots,\kappa_{2\ell}$ be the eigenvalues of
$A|_{\mathfrak k}$, listed with algebraic multiplicity, and put
$\kappa_I=\sum_{i\in I}\kappa_i$ for $I\subseteq\{1,\ldots,2\ell\}$.
Then $V^q$ has a filtration by $\mathfrak b^{0,1}$-submodules whose
successive quotients are the modules $\C_{-\kappa_Ie^0}$ with $|I|=q$.
If $A|_{\mathfrak k}$ is diagonalizable, then
\[
V^q\cong\bigoplus_{|I|=q}\C_{-\kappa_Ie^0},\qquad
\mathcal H^q\cong\bigoplus_{|I|=q}\mathcal E_{\kappa_I}.
\]
Consequently,~\eqref{eq:comparison-hypothesis} holds provided that,
for all $0\le a,b\le k$ and every
$\kappa\in\{\kappa_I:I\subseteq\{1,\ldots,2\ell\}\}$, invariant forms
induce isomorphisms
\begin{equation}\label{eq:character-hypothesis}
H^{a,b}_{\db}(\mathfrak b;\C_{-\kappa e^0})
\xrightarrow{\ \sim\ }H^b(B,\Omega^a_B\otimes\mathcal E_\kappa).
\end{equation}
\end{lemma}

\begin{proof}
Triangularize the transpose of $A|_{\mathfrak k}$ on
$\mathfrak k_\C^*$. As in the proof of Lemma~\ref{lem:wang},
$\Psi_{\mathfrak k}$ is then triangular on wedge monomials, with
diagonal entries $\kappa_I$. Since $\mathfrak b^{0,1}$ acts by multiples
of $\Psi_{\mathfrak k}$, the corresponding flag of
$\Psi_{\mathfrak k}$-invariant subspaces is a flag of submodules with the
stated quotients. For diagonalizable $A|_{\mathfrak k}$, a wedge
eigenbasis splits it, and Lemma~\ref{lem:GM} transfers the splitting to
$\mathcal H^q$.

The construction $M\mapsto\mathcal M$ is exact. A short exact sequence of
modules gives short exact sequences of invariant and of smooth Dolbeault
complexes, compatible with the inclusion. By the resulting long exact
sequences and the five lemma, induction on the length of the flag shows
that~\eqref{eq:comparison-map} is an isomorphism for $M=V^q$ once it is
an isomorphism for the one-dimensional quotients.
\end{proof}

Since $A|_{\mathfrak k}\in\mathfrak{sp}(\mathfrak k,\sigma)$, the set of
sums $\kappa_I$ is symmetric under $\kappa\mapsto-\kappa$. It contains
$\kappa=0$, from $I=\varnothing$, for which~\eqref{eq:character-hypothesis}
is the untwisted comparison on $B$. For Model~II, the results of this
section and of Section~\ref{sec:diagonal} combine as follows.

\begin{corollary}\label{cor:modelII-quotients}
Let $A$ be diagonalizable over $\R$, let $\mathcal J$ be the Model~II structure
of~\eqref{eq:models}, and let $\mathfrak k$ and $\mathfrak b$ be as in
Section~\ref{subsec:adapted}. Assume that $\Gamma$
satisfies~\eqref{eq:lattice-hypotheses}, and
that~\eqref{eq:character-hypothesis} holds for every sum $\kappa$ of a
subfamily of $(\zeta_1,-\zeta_1,\ldots,\zeta_s,-\zeta_s)$. Then, with
$N_p$ as in Theorem~\ref{thm:master},
\[
\dim GH^{n-p}_{\db}(X)=\dim GH^{p-n}_{\db}(X)=N_p+N_{p-1}
\qquad(0\le p\le2n),
\]
and invariant forms compute $GH_\partial$, $GH_{BC}$ and $GH_A$ of $X$.
Here $B$ is a compact quotient of $G_{\mathfrak b}$, whose Lie algebra
$\g_{A_c}$ has
$\spec(A_c)=\{\lambda_1,\lambda_1,\ldots,\lambda_h,\lambda_h,\mu\}$ and
carries the Model~II complex structure.
\end{corollary}

\begin{example}\label{ex:lattice-dependence}
Let $n=2$, $h=0$ and $s=1$, with $\zeta=\zeta_1\ne0$. Since $\tr A=0$,
we have $\mu=0$. Thus $d=1$, $r_1=2$, and
\[
Ae_1=0,\qquad Ae_2=\zeta e_2,\qquad Ae_3=-\zeta e_3 .
\]
The Model~II spinor $\rho=e^{\ii e^2\wedge e^3}\wedge(e^0+\ii e^1)$ has
type $1$, with $\mathfrak k=\operatorname{span}\{e_2,e_3\}$. The algebra
$\mathfrak b$ is abelian, spanned by the images of $e_0$ and $e_1$. By
Corollary~\ref{cor:generic},
\[
\bigl(\dim GH^2_{\db},\ldots,\dim GH^{-2}_{\db}\bigr)(\g_A)
=(1,2,2,2,1).
\]

Choose $M\in SL_2(\Z)$ with $\tr M\ge3$, and let $\lambda>1$ be its larger
eigenvalue. For $t_0=\log\lambda/|\zeta|$, the restriction of $e^{t_0A}$
to $\mathfrak k$ is conjugate to $M$. It therefore preserves a lattice
$\Gamma_{\mathfrak k}\subset\mathfrak k$. For $y_0>0$, put
\[
\Gamma=\Z t_0\ltimes\bigl(\Z y_0e_1\oplus\Gamma_{\mathfrak k}\bigr)<G_A .
\]
It satisfies~\eqref{eq:lattice-hypotheses} with
$\Gamma_{\mathfrak b}=\Z t_0\oplus\Z y_0$. The base $B$ is the elliptic
curve $\C/(\Z t_0+\ii\Z y_0)$, with coordinate $z=x+\ii y$, $e^0=\d x$
and $e^1=\d y$. The sums $\kappa_I$ lie in $\{0,\zeta,-\zeta\}$. Thus
$\mathcal H^0$ and $\mathcal H^2$ are holomorphically trivial, and
$\mathcal H^1\cong\mathcal E_\zeta\oplus\mathcal E_{-\zeta}$.

For $\kappa=0$,~\eqref{eq:character-hypothesis} holds, as on every
complex torus. For $\kappa=\pm\zeta$, the invariant side vanishes:
$\mathfrak b^{0,1}$ is spanned by $\bar w=\tfrac12(e_0+\ii e_1)$, which
acts on $\bigwedge^a\mathfrak b^{*1,0}\otimes\C_{-\kappa e^0}$ by the
nonzero scalar $-\kappa/2$. On the other hand,
$(e^0)^{0,1}=\tfrac12\d\bar z$ and
$\db_Be^{\ii\kappa y}=-\tfrac\kappa2e^{\ii\kappa y}\d\bar z$. Hence
\[
\db_\kappa f=e^{-\ii\kappa y}\,\db_B\bigl(\underbrace{e^{\ii\kappa y}f}_{=:g}\bigr).
\]
For a function $f$ on $B$, the function $g=e^{\ii\kappa y}f$ on $\C$
satisfies $g(z+t_0)=g(z)$ and $g(z+\ii y_0)=e^{\ii\kappa y_0}g(z)$. If
$g$ is holomorphic, $|g|$ is doubly periodic, so $g$ is constant, and it
vanishes unless $e^{\ii\kappa y_0}=1$. Indeed,
\begin{align*}
    C=g(z+\ii y_0)=e^{\ii\kappa y_0}g(z)=e^{\ii\kappa y_0}C\implies C(1-e^{\ii \kappa y_0})=0.
\end{align*}
Thus $\mathcal E_\kappa$ is
holomorphically trivial if $\kappa y_0\in2\pi\Z$, with frame
$e^{-\ii\kappa y}$. Otherwise, it is a nontrivial line bundle of degree
zero on an elliptic curve, and all its cohomology vanishes.

If $\zeta y_0\notin2\pi\Z$, Lemma~\ref{lem:characters} and
Theorem~\ref{thm:invariance} give $GH_{\db}(X)\cong GH_{\db}(\g_A)$.
If $\zeta y_0\in2\pi\Z$, then~\eqref{eq:character-hypothesis} fails
for $\kappa=\pm\zeta$. In~\eqref{eq:SS-splitting}, $E_2(Q)$ vanishes in
the rows $q=0,2$ and equals the row $q=1$ of $E_2$. By
\eqref{eq:second-pages}, this row has dimensions $2,4,2$ for
$p=0,1,2$. A spectral sequence with a single nonzero row degenerates, so
\[
\bigl(\dim GH^2_{\db},\ldots,\dim GH^{-2}_{\db}\bigr)(X)
=(1,4,6,4,1).
\]
Thus, for the same $\g_A$ and $\mathcal J$, invariant forms compute
$GH_{\db}(X)$ for some lattices but not for others. Notice that the dependence on the lattice is not a surprise; the Dolbeault cohomology of complex solvmanifolds is also known to vary with the choice of $\Gamma$, as in \cite{Kasuya13}*{Example 1}.
\end{example}

\bibliography{Pedro}
\end{document}